\documentclass[11pt]{article}
\usepackage{amssymb,amsmath}
\usepackage{mathrsfs}
\usepackage{latexsym}
\usepackage{color}
\usepackage{amssymb}
\usepackage{amsmath}
\usepackage{amsthm}

\usepackage{amsmath,amsfonts,amsthm,amssymb,amscd,color,xcolor,mathrsfs,verbatim,microtype}

\usepackage{graphicx,eurosym}
\usepackage{hyperref}
\usepackage[applemac]{inputenc}
\usepackage[cyr]{aeguill}
\colorlet{darkblue}{blue!50!black}
\hypersetup{
    colorlinks,%
    citecolor=blue,%
    filecolor=red,%
    linkcolor=darkblue,%
    urlcolor=blue,%
    pdfnewwindow=true,%
    pdfstartview={FitH}
}

\allowdisplaybreaks

\catcode`!=11
\let\!int\int \def\int{\displaystyle\!int}
\let\!lim\lim \def\lim{\displaystyle\!lim}
\let\!sum\sum \def\sum{\displaystyle\!sum}
\let\!sup\sup \def\sup{\displaystyle\!sup}
\let\!inf\inf \def\inf{\displaystyle\!inf}
\let\!cap\cap \def\cap{\displaystyle\!cap}
\let\!max\max \def\max{\displaystyle\!max}
\let\!min\min \def\min{\displaystyle\!min}
\catcode`!=12

\newtheorem{theorem}{\bf Theorem}[section]
\newtheorem{lemma}{\bf Lemma}[section]

\newtheorem{proposition}{\bf Proposition}[section]

\newtheorem{remark}{\bf Remark}[section]

\catcode`@=11
\@addtoreset{equation}{section}
\catcode`@=12

\begin{document}

\title{Ergodicity for the randomly forced Korteweg-de Vries-Burgers equation}
\author{Peng Gao
\\[2mm]
\small School of Mathematics and Statistics, and Center for Mathematics
\\
\small and Interdisciplinary Sciences, Northeast Normal University,
\\
\small Changchun 130024,  P. R. China
\\[2mm]
\small Email: gaopengjilindaxue@126.com }
\date{\today}
\maketitle

\vbox to -13truemm{}

\begin{abstract}
Our goal in this paper is to investigate ergodicity of the randomly forced Korteweg-de Vries-Burgers(KdVB) equation driven by non-additive white noise. Under reasonable conditions, we show that exponential ergodicity for KdVB equation driven by a space-time localised noise and ergodicity for KdVB equation driven by a multiplicative white noise. Our proof is based on some newly developed analytical properties for KdVB equation, such as Carleman estimate, truncated observability inequality, Foia\c{s}-Prodi estimate. Combining these analytical properties with coupling method and asymptotic coupling method, we can investigate the long time behavior of randomly forced KdVB equation.
\\[6pt]
{\sl Keywords: ergodicity; Korteweg-de Vries-Burgers equation; Foia\c{s}-Prodi estimate; Carleman estimate; degenerate random force}
\\
{\sl 2020 Mathematics Subject Classification: 60H15, 35R60, 37A25}
\end{abstract}
\tableofcontents
\setcounter{section}{0}

\section{Introduction}
The motion of long, unidirectional, weakly nonlinear water waves on a
channel can be described, as is well known, by the Korteweg de Vries(KdV) equation. This equation has been proposed as a model for small-amplitude, long
waves in many different physical systems. It incorporates effects of dispersion and
of nonlinear convection, yielding good qualitative predictions of various observable
phenomena. However, in many real situations, to effect quantitative agreement of predictions
with experimentally obtained data, energy dissipation mechanisms may need, accounted
for Korteweg-de Vries-Burgers (KdVB) equation through the term $-u_{xx}+u$, namely,
$$u_{t}+u_{xxx}-u_{xx}+u+uu_{x}=0.$$
The KdVB equation is considered as a simple model displaying the features of dissipation, dispersion and nonlinearity. It also arises in many physical applications such as propagation waves in elastic tube filled with a viscous fluid and weakly nonlinear plasma waves with certain dissipative effects(see for instance \cite{Background1,Background2}). The soliton propagation in the
random weakly viscous media can also be studied in the framework of the forced KdVB equation. In recent years, the KdVB equation has attracted the many researchers' attention(see for instance \cite{Chen2017,r1,r2,r3,r5,A1}).
\par
In the study of water waves, when the surface of the fluid is submitted to a non
constant pressure, or when the bottom of the layer is not flat, a forcing
term has to be added to the equation. This term is given by the
gradient of the exterior pressure or of the function whose graph defines the
bottom. The soliton propagation in the random weakly viscous media or in the random field can also be studied in the framework of the random forced KdVB equation (see for instance \cite{r4,Background3,Background4,GP2023non}).
\par
In this paper, we investigate the long time behavior for the randomly forced KdVB equation. To be specific, we consider the following randomly forced KdVB equation on $\mathbb{T}=\mathbb{R}/2\pi\mathbb{Z}$:
\begin{eqnarray}\label{1}
\begin{array}{l}
\left\{
\begin{array}{lll}
u_{t}+u_{xxx}-u_{xx}+u+uu_{x}=h+\eta\\
u(x,0)=u_{0}
\end{array}
\right.
\end{array}
\begin{array}{lll}
\textrm{in}~~\mathbb{T},\\
\textrm{in}~~\mathbb{T}.
\end{array}
\end{eqnarray}
where $h=h(t,x)$ is a given function and $\eta$ is a stochastic process which will be specified later.
\par
Motivated from both physical and mathematical standpoints, an important mathematical question arises:
\par
~~
\par
\rm
\textbf{What is the asymptotic behavior of $u(t)$ as $t\rightarrow+\infty$~?}
\rm
\par
~~
\par
As we know, ergodicity is an effective tool to describe long time behavior for the randomly forced PDEs. In the last decades there have been many papers on the subject
of ergodicity for partial differential equations(PDEs) with random forcing, we refer the reader
to \cite{M1,Fla1,Deb1} and the book \cite{KS12} for a detailed discussion of the literature in this direction. The large majority of the works
concern PDEs driven by an additive white noise, whereas the papers
concerning \textit{non-additive white noise} (e.g. multiplicative-type noise or space-time localised noise) are much scarcer. Up to now, there is less work on ergodicity for KdV type equation, \cite{Glat2021} establishes its mixing property in the case of additive white noise.
\par
Different from the previous results, in this paper, we investigate ergodicity for the randomly forced KdVB equation driven by a \textit{space-time localised noise} and a \textit{multiplicative white noise}. More precisely, we establish the following main results:
\par
(i) \textit{Exponential ergodicity for KdVB equation driven by a space-time localised noise} (see Theorem \ref{MT}). Space-time localised noise is an important kind of noises
in engineering and physics, this kind of noise is both degenerate in Fourier space and physical space (see \cite{S1,S2}), this motivates us to apply the coupling method and control method to ergodicity problem. This framework is successfully used to establish ergodicity for Navier-Stokes system with a space-time localised noise (see \cite{S1,S2}), and it also works for dissipative PDEs driven by a degenerate bounded noise (see \cite{K2020GAFA,K2020JEPM,GARMA}). In order to apply this framework to our problem, we need a new Carleman estimate for KdVB equation, it plays a key role in our proof. However, the classical method for Carleman estimate is hard to be applied to KdVB equation, here we develop some new technique to KdVB equation. Based on the new Carleman estimate, we can derive a truncated observability inequality for KdVB equation and to apply coupling method to infer exponential ergodicity. The Carleman estimate for KdVB equation we obtain are an interesting result in themselves, and we hope to use it also to obtain controllability and quantitative decay results.
\par
(ii) \textit{Ergodicity for KdVB equation driven by a multiplicative white noise} (see Theorem \ref{MT2} and Theorem \ref{MT3}). In this paper, we prove uniqueness of the invariant measure and asymptotic stability for KdVB equation with a multiplicative white noise. Since many methods on PDEs driven by an additive stochastic forcing term don't work in the case of multiplicative white noise, this motivates us to apply asymptotic coupling method.
Asymptotic coupling method provides a flexible and intuitive framework for proving the uniqueness of invariant measures for a variety of PDEs with an additive white noise, see \cite{AC1,AC2,AC3,AC5}. Differently from the previous results, here we deal with more general noises. The covariance operator either is bounded or satisfies a sublinear or a linear growth condition. This case is also considered for Navier-Stokes system in \cite{AC4}. Asymptotic coupling method is heavily based on Foia\c{s}-Prodi estimate for PDEs. In our proof, Foias-Prodi estimates in expectation is a crucial tool. Due to the multiplicative white noise, Foia\c{s}-Prodi estimate in our case is harder than the case of PDEs with an additive white noise. Here, we develop some new techniques to overcome difficulties. We derive the Foias-Prodi estimate in expectation for the KdVB equation and show that it is in fact the crucial ingredient to readapt the asymptotic coupling method of \cite{AC1} and \cite{AC2} to
infer uniqueness of the invariant measure and asymptotic stability in the presence of multiplicative white noise.
\par~~
\par

The rest of the paper is organized as follows. In Section 2, we introduce the mathematical setting and main results in this paper. We establish a new Carleman estimate for the KdVB equation in Section 3. In Section 4, we introduce the coupling method and an abstract criterion for the proof of Theorem \ref{MT}. A new Foia\c{s}-Prodi estimate for KdVB equation is established in Section 5. In Section 6, we prove Theorem \ref{MT2} and Theorem \ref{MT3}.

\section{Main results}
\subsection{Mathematical setting}
Let $X$ be a Polish space with a metric $d_{X}(u,v)$, the Borel $\sigma$-algebra on $X$ is denoted by $\mathcal{B}(X)$ and the set of Borel
probability measures by $\mathcal{P}(X).$ $C_{b}(X)$ is the space of continuous functions $f:X\rightarrow \mathbb{C}$ endowed with the
norm $\|f\|_{\infty}=\sup_{u\in X}|f(u)|.$ $B_{X}(R)$ stands for the ball in $X$ of radius $R$ centred at zero.
We write $C(X)$ when $X$ is compact. $L_{b}(X)$ is the space of functions $f\in C_{b}(X)$ such that
$$\|f\|_{L(X)}=\|f\|_{\infty}+\sup\limits_{u\neq v}\frac{|f(u)-f(v)|}{d_{X}(u,v)}<\infty.$$
The dual-Lipschitz metric on $\mathcal{P}(X)$ is defined by
$$\|\mu_{1}-\mu_{2}\|^{*}_{L(X)}=\sup_{\|f\|_{L(X)}\leq 1}|\langle f,\mu_{1}\rangle-\langle f,\mu_{2}\rangle|,~~\mu_{1},\mu_{2}\in \mathcal{P}(X),$$
where $\langle f,\mu \rangle=\int_{X}f(u)\mu(du).$
\par
We denote by $L^{2}(\mathbb{T})(=H^{0}(\mathbb{T}))$ the space of all Lebesgue square integrable functions on $\mathbb{T}$. The inner product on $L^{2}(\mathbb{T})$ is
$(u,v)=\int_{\mathbb{T}}uvdx,$ for any $u,v\in L^{2}(\mathbb{T}).$ The norm on $L^{2}(\mathbb{T})$ is $\|u\|=( u,u )^{\frac{1}{2}},$ for any $u\in L^{2}(\mathbb{T}).$ The definition of $H^{s}(\mathbb{T})$ can be found in \cite{Lions}, the norm on $H^{s}(\mathbb{T})$ is $\|\cdot\|_{H^{s}}.$
We define the spaces
$$X_{i}=C([0,T];H^{i}(\mathbb{T}))\cap L^{2}(0,T;H^{i+1}(\mathbb{T}))~~(i=0,1,2)$$
equipped with their natural norms.
\par
Set $H:=L^{2}(\mathbb{T}).$ Let $\{e_{i}\}$ be an orthonormal basis in $H$ formed of the eigenfunctions of the Laplacian in $\mathbb{T}$,
$\{\lambda_{i}\}$ be the corresponding sequence of eigenvalues for $-\partial_{xx}+1$ in $\mathbb{T}$, and $P_{N}$ be the orthogonal projection in $H$ on the vector space
$$H_{N}:=span\{e_{1},\cdots,e_{N}\}.$$
\par
Set $D_{T}:=\mathbb{T}\times(0,T)$. Let $Q\subset D_{T}$ be an open set, $\{\varphi_{i}\}\subset H^{1}(Q)$ be an orthonormal basis in $L^{2}(Q)$ formed of the eigenfunctions of $-\partial_{xx}-\partial_{tt}+1$ in $Q$,
$\{\alpha_{i}\}$ be the corresponding sequence of eigenvalues for $-\partial_{xx}-\partial_{tt}+1$ in $Q$, and $\Pi_{N}$ be the orthogonal projection in $L^{2}(Q)$ on the vector space
$$\mathcal{E}_{N}:=span\{\varphi_{1},\cdots,\varphi_{N}\}.$$
Let $\psi_{i}=\chi\varphi_{i}$ be linearly independent, where $\chi\in C_{0}^{\infty}(Q)$ is a non-zero function. Extending the functions $\psi_{i}$ by zero outside $Q,$ we may regard them as elements of $H^{1}_{0}(D_{T})$.
\par
Given two Banach spaces $E$ and $F,$ we denote by $L(E; F)$ the space of all linear bounded operators $B : E\rightarrow F$ and abbreviate $L(E):=L(E;E).$
If $H$ and $K$ are separable Hilbert spaces, we employ the symbol $L_{HS}(H;K)$ for the space of Hilbert-Schmidt operators from $H$ to $K.$

\par
Throughout the paper, the letter $C$ denotes a positive constant whose value may change in different occasions. We will write the dependence of constant on parameters explicitly if it is essential.
\par
If we define $Au:=u_{xxx}-u_{xx}+u, B(u):=uu_{x}$, we can rewrite \eqref{1} as
\begin{eqnarray}\label{1r}
\begin{array}{l}
\left\{
\begin{array}{lll}
u_{t}+Au+B(u)=h+\eta\\
u(x,0)=u_{0}
\end{array}
\right.
\end{array}
\begin{array}{lll}
\textrm{in}~~\mathbb{T},\\
\textrm{in}~~\mathbb{T}.
\end{array}
\end{eqnarray}
This compact form is useful for our later argument.
\par~~
\par
Now, we are in a position to present the main results in this paper.

\subsection{KdVB equation driven by a space-time localised noise}
In this section, $\eta$ is a stochastic process of the form
\begin{equation}\label{39}
\begin{split}
\eta(t,x)=\sum\limits_{k=1}^{\infty}I_{k}(t)\eta_{k}((t-k+1)T,x),~~t\geq0,
\end{split}
\end{equation}
$I_{k}$ is the indicator function of the interval $((k-1)T,kT)$ and $\{\eta_{k}\}$ is a sequence of
i.i.d. random variables in $L^{2}(D_{T})$ that are continued by zero for $t\notin [0,T].$

\par
For any $k\geq1,$ the random variables $\eta_{k}$ satisfy the following condition
\par
\textbf{(DN) Structure of the noise.} The random variables $\eta_{k}$ has the form
\begin{equation*}
\begin{split}
\eta_{k}(t,x)=\sum\limits_{i=1}^{\infty}b_{i}\xi_{ik}\psi_{i}(t,x),
\end{split}
\end{equation*}
where $\xi_{ik},\psi_{i},b_{i}$ satisfy the following assumptations
\par
$\bullet$ $\xi_{ik}$ are independent scalar random variables such that $|\xi_{ik}|\leq1$ with probability $1,$ and there are non-negative functions
$p_{i}\in C^{1}(\mathbb{R})$ such that for any $i\geq1$
$$p_{i}(0)\neq0, \mathcal{D}(\xi_{ik})=p_{i}(r)dr;$$
\par
$\bullet$ $\{b_{i}\}\subset \mathbb{R}$ is a non-negative sequence such that $B:=\sum\limits_{i=1}^{\infty}b_{i}\|\psi_{i}\|_{H^{2}(Q)}<+\infty.$
\par
$\bullet$ Let $\mathcal{K}\subset L^{2}(Q)$ be the support of the law of $\eta_{k},$ which contains the origin.
\begin{remark}
The hypotheses imposed on $\eta_{k}$ imply that $\mathcal{K}$ is a compact subset in $H^{1}_{0}(Q).$ Continuing the elements of $\mathcal{K}$ by zero outside $Q,$ we may regard $\mathcal{K}$ as a compact subset of $H^{1}_{0}(D_{T}).$
\end{remark}
\par~~
\par
We introduce the following spaces $H:=L^{2}(\mathbb{T}),E:=L^{2}(D_{T})$ and define an operator $S$ as
\begin{equation*}
\begin{split}
S: H\times E&\mapsto H\\
(u_{0},f)&\mapsto u(T),
\end{split}
\end{equation*}
where $u$ is a solution to \eqref{1} with $h+\eta=f$.
If $u(t)$ is a solution of \eqref{1} with $\eta$ in \eqref{39} and denote $u_{k}=u(kT),$ it is easy to see that
\begin{equation}\label{LS1}
\begin{split}
u_{k}=S(u_{k-1},h+\eta_{k}),~~k\geq1.
\end{split}
\end{equation}
Since $\eta_{k}$ are i.i.d. random variables in $E$, \eqref{LS1} defines a homogeneous family of Markov chains in $H,$ which is denoted by $(u_{k},\mathbb{P}_{u}),u\in H.$ Let $P_{k}(u,\Gamma)$ be the transition function for $(u_{k},\mathbb{P}_{u}).$
\par~~
\par
\textbf{(AC) Approximate controllability to a given point.} There is $\bar{u}\in H$ such that for any positive constants $R$ and $\varepsilon>0$, there is an integer $l\geq1$ such that for any $v\in B_{H}(R),$ there
are $\zeta_{1},\zeta_{2},\cdots,\zeta_{l}\in \mathcal{K}$ such that
\begin{equation}\label{24}
\begin{split}
\|S_{l}(v,\zeta_{1},\zeta_{2},\cdots,\zeta_{l})-\bar{u}\|\leq \varepsilon,
\end{split}
\end{equation}
where $S_{l}(v,\zeta_{1},\zeta_{2},\cdots,\zeta_{l})$ stands for $u_{l}$ defined by \eqref{LS1} with $\eta_{l}=\zeta_{l}(1\leq l\leq k)$ and $u_{0}=v.$
\par
~~
\par
Now, we are in a position to present the following result in this paper.
\begin{theorem}\label{MT}
Let $h\in H^{1}_{loc}(\mathbb{R}_{+}\times \mathbb{T})$ be $T-$periodic in time, Conditions (DN) and (AC) hold.
Then there exists an integer $N\geq1$ such that if
\begin{equation}\label{41}
\begin{split}
b_{i}\neq0,~i=1,2,\cdots,N,
\end{split}
\end{equation}
holds, then there is a unique stationary measure $\mu\in \mathcal{P}(H)$ and positive numbers $C,\sigma$ such that, for any $u_{0}\in H,$ the solution $u$ of \eqref{1} with $\eta$ in \eqref{39} satisfies
\begin{equation*}
\begin{split}
\|P_{k}(u_{0},\cdot)-\mu\|_{L}^{*}\leq C(1+\|u_{0}\|^{2})e^{-\sigma k},~k\geq 0.
\end{split}
\end{equation*}
\end{theorem}

\subsection{KdVB equation driven by a multiplicative white noise}
In this section, $\eta$ is a multiplicative noise of the form
\begin{equation}\label{43}
\begin{split}
\eta(t)=g(u)\frac{\partial}{\partial t}W,~~t\geq0,
\end{split}
\end{equation}
where $W$ is a $U$-cylindrical Wiener process defined on a filtered probability space
$(\Omega,\mathcal{F},\mathcal{F}_{t},\mathbb{P})$, $U$ is a separable real Hilbert space (see \cite{DaZ1}).
\par
We introduce the following assumptions on the operator $g$:
\par
\textbf{(g1)} $g:H\rightarrow L_{HS}(U,H)$ is a Lipschitz continuous operator, i.e., there exists a constant $L_{g}>0$ such that
$$\|g(u_{1})-g(u_{2})\|_{HS}\leq L_{g}\|u_{1}-u_{2}\|,~\forall u_{1},u_{2}\in H.$$
\par
\textbf{(g2)} There exist constant $K_{i}>0(i=1,2,3),L_{j}>0(j=2,3)$ and $\varrho\in (0,1)$ such that for any $u\in H$, it holds that
\begin{equation*}
\begin{split}
\left\{
\begin{array}{lll}
\textbf{(g2)}(i)~~~\|g(u)\|_{HS}\leq K_{1},\\
\textbf{(g2)}(ii)~~\|g(u)\|_{HS}\leq K_{2}+L_{2}\|u\|^{\varrho},\\
\textbf{(g2)}(iii)~ \|g(u)\|_{HS}\leq K_{3}+L_{3}\|u\|.
\end{array}
\right.
\end{split}
\end{equation*}
\par
\textbf{(g3)} There exists a measurable map $f:H\rightarrow L(H, U)$ such that $\sup\limits_{u\in H}\|f(u)\|_{L(H, U)}<+\infty$ and
$$g(u)f(u)=P_{M},~\forall u\in H$$
for a positive integer $M.$
\par~~
\par
By the similar arguments as in \cite{AC4}, we can establish the well posedness of system \eqref{1} with $\eta$ in \eqref{43} under the conditions (g1) and (g2).
By this fact we can introduce Markov transition semigroup $\{\mathcal{P}_{t}\}_{t\geq 0}$. Using a Krylov-Bogoliubov argument together with tightness, \eqref{1} driven by the following random forces always admits at least one invariant measure, see \cite{DaZ1}. Indeed, when Conditions (g1)-(g3) hold with $L_{3}<1,$ we prove
that Markov transition semigroup $\{\mathcal{P}_{t}\}_{t\geq 0}$ is Feller and admits at least one invariant measure as the similar argument in \cite[Proposition 5.2]{AC4}.
\par~~
\par
Now, we are in a position to present the following results.
\begin{theorem}\label{MT2}
Let $L_{3}<1,$ $h=h(x)\in H$, Conditions (g1)-(g3) hold.
Then there exists an integer $N_{0}\geq1$ such that if (g3) holds for some $M\geq N_{0}$,
then $\{\mathcal{P}_{t}\}_{t\geq 0}$ possesses a unique ergodic invariant measure $\mu\in \mathcal{P}(H),$
where $\{\mathcal{P}_{t}\}_{t\geq 0}$ is the Markov semigroup corresponding to \eqref{1} with $\eta$ in \eqref{43} on $H$.
\end{theorem}

\begin{theorem}\label{MT3}
Let $L_{3}<\frac{1}{\sqrt{5}},$ $h=h(x)\in H$, Conditions (g1)-(g3) hold.
Then there exists an integer $N_{0}\geq1$ such that if (g3) holds for some $M\geq N_{0}$,
then $\{\mathcal{P}_{t}\}_{t\geq 0}$ possesses a unique ergodic invariant measure $\mu\in \mathcal{P}(H)$ and
\begin{equation*}
\begin{split}
\lim\limits_{t\rightarrow\infty}\|\mathcal{P}_{t}^{*}\delta_{u_{0}}-\mu\|_{L}^{*}=0,~~\forall u_{0}\in H.
\end{split}
\end{equation*}
\end{theorem}

\section{A new global Carleman estimate for linear KdVB equation}

\subsection{Carleman estimate for linear KdVB equation}
In this paper, Carleman estimate is an important tool to establish the mixing for KdVB equation. Based on it, we develop a general
strategy and framework to deal with a space-time localised noise. In order to prove Theorem \ref{MT}, we combine Carleman estimate with coupling method. It is well known that Carleman estimate is an $L^{2}$-weighted estimate with large parameter for a solution
to a PDE and it is one of the major tools used in the study of unique continuation, observability and controllability problems for various kinds of PDEs (see \cite{Y1,Ta1,F3,K3} and references therein). In this section, we will establish a new global Carleman estimate for linear complex KdVB equation on the tours, which gives a connection between solutions on the whole domain and on an arbitrarily given subdomain.
\par
Let us consider the following linear KdVB equation
\begin{equation}\label{C1}
\begin{cases}
-v_{t}-v_{xxx}-v_{xx}+av_{x}+bv=g\quad &\textrm{in}~\mathbb{T}\times(0,T),
\\v(x,T)=v_{T}(x)\quad &\textrm{in}~\mathbb{T}.
\end{cases}
\end{equation}
\par
Let $\omega=(l_{1},l_{2})$ with $0<l_{1}<l_{2}<2\pi$. Set $D_{\omega}=\omega\times(0,T)$.
Let us pick a weight function $\psi\in C^{\infty}(\mathbb{T})$ with
\begin{equation}\label{C2}
\psi>0~~\textrm{in}~\mathbb{T},~~|\psi^{\prime}|>0,~\psi^{\prime\prime}<0~~\textrm{in}~\mathbb{T}\backslash \omega,~~2\max\limits_{x\in\mathbb{T}}\psi(x)<3\min\limits_{x\in\mathbb{T}}\psi(x).
\end{equation}
An example of the weight function $\psi$ is constructed as following. Pick a function $\phi\in C^{\infty}([0,2\pi])$ is defined on $[0,2\pi]\backslash(l_{1},l_{2})$ as
\begin{equation*}
\phi(x)=
\begin{cases}
-(x+b)^{2}+b^{2} &x\in(0,l_{1}),
\\-(x-2\pi+b)^{2}+b^{2}  &x\in(l_{2},2\pi),
\end{cases}
\end{equation*}
where $b>2\pi-l_{2}$. We define $\psi(x)=\phi(x)+C_{0}$ in $[0,2\pi]$ with $C_{0}>2\hat{\phi}-3\breve{\phi}$, where
$\hat{\phi}:=\max_{x\in [0,2\pi]}\phi(x), \breve{\phi}:=\min_{x\in [0,2\pi]}\phi(x)$. Then, we do a periodic extension of $\psi$
from $[0,2\pi]$ to the entire torus $\mathbb{T}.$ It is easy to see that $\psi$ satisfies \eqref{C2}.
\par
Let us introduce the following functions
\begin{equation*}
\begin{split}
&\xi(t)=\frac{1}{t(T-t)},~~\varphi(x,t)=\psi(x)\xi(t),\\
&\hat{\varphi}(t)=\max\limits_{x\in\mathbb{T}}\psi(x)\xi(t),~\check{\varphi}(t)=\min\limits_{x\in\mathbb{T}}\psi(x)\xi(t).
\end{split}
\end{equation*}
It is easy to see that $2\hat{\varphi}(t)<3\check{\varphi}(t)$ for $t\in(0,T)$, and
\begin{equation*}
\begin{split}
&\partial_{x}^{i}\partial^{j}_{t}\varphi\leq C\varphi^{j+1}~~\textrm{in}~D_{T},~~i=0,1,2,3,4,~j=0,1,\\
& \xi^{k}e^{-\alpha \xi}\leq C,~~\textrm{for}~k\in\mathbb{N},\alpha>0.
\end{split}
\end{equation*}
\par~~
\par
Now, we are in a position to present the Carleman estimate for the linear KdVB equation.
\begin{theorem}\label{CT1}
Let $a,b\in X_{0}$. Then, there exist two positive constants $s_{0}$ and $C$ such that for any $v_{T}\in L^{2}(\mathbb{T})$, $g\in L^{2}(D_{T})$ and any $s\geq s_{0}$, the corresponding solution to \eqref{C1} satisfies:
\begin{equation*}
\int_{D_{T}}(s\xi v_{xx}^{2}+s^{3}\xi^{3}v_{x}^{2}+s^{5}\xi^{5}v^{2})e^{-4s\hat{\varphi}}dxdt
\leq C\Big(s^{5}\int_{D_{T}}g^{2}e^{-2s\hat{\varphi}}dxdt+s^{5}\int_{D_{\omega}}\xi^{5}e^{-6s\check{\varphi}+2s\hat{\varphi}}v^{2}dxdt\Big).
\end{equation*}
\end{theorem}
\par
To prove Theorem \ref{CT1}, we need the following two lemmas.

\begin{lemma}\label{CL1}
Let $a,b\in X_{0}$, $h\in L^{2}(0,T;H^{1}(\mathbb{T}))$. There exists a positive constant $C=C(\|a\|_{X_{0}},\|b\|_{X_{0}})$ such that the solution $y$ of the system
\begin{equation}\label{C12}
\begin{cases}
y_{t}+y_{xxx}-y_{xx}+ay_{x}+by=h \quad &\textrm{in}~D_{T},
\\y(x,0)=u_{0}\quad &\textrm{in}~\mathbb{T},
\end{cases}
\end{equation}
satisfies
\begin{equation*}
\|y\|_{X_{i}}\leq C(\|u_{0}\|_{i}+\|h\|_{L^{2}(0,T;H^{i-1}(\mathbb{T}))})~~\textrm{for}~~i=0,1,2.
\end{equation*}
\end{lemma}
\par

We introduce the operator $L$ defined by $Lq=q_{t}+q_{xxx}$ with its domain
$$\mathcal{D}(L)=L^{2}(0,T;H^{3}(\mathbb{T}))\cap H^{1}(0,T;L^{2}(\mathbb{T})).$$
\begin{lemma}\label{CL2}
There exists a positive constant $s_{1}$ such that for any $q\in \mathcal{D}(L)$ and $s\geq s_{1}$, we have
\begin{equation}\label{C8}
\begin{split}
&\int_{D_{T}}(s\varphi w_{xx}^{2}+s^{3}\varphi^{3}w_{x}^{2}+s^{5}\varphi^{5}w^{2})dxdt\\
\leq&C\int_{D_{T}}|Lq|^{2}e^{-2s\varphi}dxdt+C\int_{D_{\omega}}(s\varphi w_{xx}^{2}+s^{3}\varphi^{3}w_{x}^{2}+s^{5}\varphi^{5}w^{2})dxdt,
\end{split}
\end{equation}
where $w=e^{-s \varphi}q$.
\end{lemma}
\par
Now we can give the proof of Theorem \ref{CT1}.
\begin{proof}[Proof of Theorem \ref{CT1}]
We decompose the solution $v$ of \eqref{C1}. In other words, let us consider the following systems
\begin{equation}\label{C6}
\begin{cases}
-q_{t}-q_{xxx}-q_{xx}+aq_{x}+bq=-\rho_{t}v\quad &\textrm{in}~D_{T},
\\q(x,T)=0\quad &\textrm{in}~\mathbb{T},
\end{cases}
\end{equation}
and
\begin{equation}\label{C7}
\begin{cases}
-z_{t}-z_{xxx}-z_{xx}+az_{x}+bz=\rho g\quad &\textrm{in}~D_{T},
\\z(x,T)=0\quad &\textrm{in}~\mathbb{T},
\end{cases}
\end{equation}
where $\rho(t)=e^{-s\hat{\varphi}(t)}$. By uniqueness for the linear KdVB equation, we have
\begin{equation}\label{C9}
\rho v=q+z.
\end{equation}
\par

\par
Considering the definition of operator $L$ and system \eqref{C6}, we can deduce that
\begin{equation*}
\begin{split}
e^{-s\varphi}Lq=&e^{-s\varphi}(\rho_{t}v-q_{xx}+aq_{x}+bq)\\
=&e^{-s\varphi}\rho_{t}v-w_{xx}+(-2s\varphi_{x}+a)w_{x}+(-s\varphi_{xx}-s^{2}\varphi_{x}^{2}+as\varphi_{x}+b)w.
\end{split}
\end{equation*}
It is not difficult to obtain that
\begin{equation*}
\begin{split}
\int_{D_{T}}(-2s\varphi_{x}+a)^{2}w^{2}_{x}dxdt
\leq& Cs^{2}\int_{D_{T}}\varphi^{2}w_{x}^{2}dxdt+C\int_{0}^{T}\|a\|^{2}\|w_{x}\|^{2}_{\infty}dt\\
\leq& Cs^{2}\int_{D_{T}}\varphi^{2}w_{x}^{2}dxdt+C\|a\|_{X_{0}}^{2}\int_{D_{T}}(w_{x}^{2}+w_{xx}^{2})dxdt,
\end{split}
\end{equation*}
\begin{equation*}
\begin{split}
&\int_{D_{T}}(-s\varphi_{xx}-s^{2}\varphi_{x}^{2}+as\varphi_{x}+b)^{2}w^{2}dxdt\\
\leq& Cs^{4}\int_{D_{T}}\varphi^{4}w^{2}dxdt+C s^{2}\int_{0}^{T}\|a\|^{2}\|\varphi_{x}w\|^{2}_{\infty}dt+C\int_{0}^{T}\|b\|^{2}\|w\|^{2}_{\infty}dt\\
\leq& Cs^{4}\int_{D_{T}}\varphi^{4}w^{2}dxdt+Cs^{2}\|a\|_{X_{0}}^{2}\int_{D_{T}}\varphi^{2}(w^{2}+w_{x}^{2})dxdt+C\|b\|_{X_{0}}^{2}\int_{D_{T}}(w^{2}+w_{x}^{2})dxdt.
\end{split}
\end{equation*}
Thus we have
\begin{equation*}
\begin{split}
\int_{D_{T}}|Lq|^{2}e^{-2s\varphi}dxdt\leq& Cs^{2}\int_{D_{T}}\xi^{4}e^{-2s\hat{\varphi}}e^{-2s\varphi}v^{2}dxdt\\
&+C(\|a\|_{X_{0}}^{2}+\|b\|_{X_{0}}^{2}+1)\int_{D_{T}}(w_{xx}^{2}+s^{2}\varphi^{2}w_{x}^{2}+s^{4}\varphi^{4}w^{2})dxdt.
\end{split}
\end{equation*}
Substituting this inequality into \eqref{C8}, picking $s\gg1$, we arrive to
\begin{equation*}
\begin{split}
&\int_{D_{T}}(s\varphi w_{xx}^{2}+s^{3}\varphi^{3}w_{x}^{2}+s^{5}\varphi^{5}w^{2})dxdt\\
\leq&Cs^{2}\int_{D_{T}}\xi^{4}e^{-2s\hat{\varphi}}e^{-2s\varphi}v^{2}dxdt+C\int_{D_{\omega}}(s\varphi w_{xx}^{2}+s^{3}\varphi^{3}w_{x}^{2}+s^{5}\varphi^{5}w^{2})dxdt.
\end{split}
\end{equation*}
Replacing $w$ by $e^{-s\varphi}q$ and noting \eqref{C9}, we have
\begin{equation*}
\begin{split}
&\int_{D_{T}}(s\varphi q_{xx}^{2}+s^{3}\varphi^{3}q_{x}^{2}+s^{5}\varphi^{5}q^{2})e^{-2s\varphi}dxdt\\
\leq&Cs^{2}\int_{D_{T}}\xi^{4}e^{-2s\varphi}(q^{2}+z^{2})dxdt+C\int_{D_{\omega}}(s\varphi q_{xx}^{2}+s^{3}\varphi^{3}q_{x}^{2}+s^{5}\varphi^{5}q^{2})e^{-2s\varphi}dxdt.
\end{split}
\end{equation*}
Taking $s_{0}$ sufficiently large, for $s\geq s_{0}$, it follows that
\begin{equation}\label{C10}
\begin{split}
&\int_{D_{T}}(s\xi q_{xx}^{2}+s^{3}\xi^{3}q_{x}^{2}+s^{5}\xi^{5}q^{2})e^{-2s\hat{\varphi}}dxdt\\
\leq&Cs^{2}\|z\|^{2}_{L^{2}(0,T;L^{2}(\mathbb{T}))}+C\int_{D_{\omega}}(s\xi q_{xx}^{2}+s^{3}\xi^{3}q_{x}^{2}+s^{5}\xi^{5}q^{2})e^{-2s\check{\varphi}}dxdt.
\end{split}
\end{equation}
\par
Using the fact that $H^{1}(\omega)=(H^{3}(\omega),L^{2}(\omega))_{2/3,2}$ and $H^{2}(\omega)=(H^{3}(\omega),L^{2}(\omega))_{1/3,2}$, we can obtain that
\begin{equation*}
\begin{split}
s^{3}\int_{D_{\omega}}\xi^{3}q_{x}^{2}e^{-2s\check{\varphi}}dxdt\leq& Cs^{3}\int_{0}^{T}\xi^{3}\|q\|^{4/3}_{L^{2}(\omega)}\|q\|^{2/3}_{H^{3}(\omega)}e^{-2s\check{\varphi}}dt\\
\leq& Cs^{5}\int_{0}^{T}\xi^{5}\|q\|^{2}_{L^{2}(\omega)}e^{-3s\check{\varphi}+s\hat{\varphi}}dt+
C s^{-1}\int_{0}^{T}\xi^{-1}\|q\|^{2}_{H^{3}(\omega)}e^{-2s\hat{\varphi}}dt,\\
s\int_{D_{\omega}}\xi q_{xx}^{2}e^{-2s\check{\varphi}}dxdt\leq& Cs\int_{0}^{T}\xi\|q\|^{2/3}_{L^{2}(\omega)}\|q\|^{4/3}_{H^{3}(\omega)}e^{-2s\check{\varphi}}dt\\
\leq& Cs^{5}\int_{0}^{T}\xi^{5}\|q\|^{2}_{L^{2}(\omega)}e^{-6s\check{\varphi}+4s\hat{\varphi}}dt+
C s^{-1}\int_{0}^{T}\xi^{-1}\|q\|^{2}_{H^{3}(\omega)}e^{-2s\hat{\varphi}}dt.
\end{split}
\end{equation*}
Substituting this estimates into \eqref{C10}, we infer that
\begin{equation}\label{C11}
\begin{split}
&\int_{D_{T}}(s\xi q_{xx}^{2}+s^{3}\xi^{3}q_{x}^{2}+s^{5}\xi^{5}q^{2})e^{-2s\hat{\varphi}}dxdt\\
\leq&Cs^{2}\|z\|^{2}_{L^{2}(0,T;L^{2}(\mathbb{T}))}+Cs^{5}\int_{0}^{T}\xi^{5}\|q\|^{2}_{L^{2}(\omega)}e^{-6s\check{\varphi}+4s\hat{\varphi}}dt+
C s^{-1}\int_{0}^{T}\xi^{-1}\|q\|^{2}_{3}e^{-2s\hat{\varphi}}dt.
\end{split}
\end{equation}
In order to estimate the last term in the right hand side, we define
\begin{equation*}
\hat{q}=\hat{\rho}q~~\textrm{with}~~\hat{\rho}(t):=s^{-1/2}\xi^{-1/2}(t)e^{-s\hat{\varphi}(t)}.
\end{equation*}
From \eqref{C6}, we can see that $\hat{q}$ is the solution of the following system
\begin{equation*}
\begin{cases}
-\hat{q}_{t}-\hat{q}_{xxx}-\hat{q}_{xx}+a\hat{q}_{x}+b\hat{q}=-\hat{\rho}\rho_{t}v-\hat{\rho}_{t}q:=\hat{g}\quad &\textrm{in}~D_{T},
\\\hat{q}(x,T)=0\quad &\textrm{in}~\mathbb{T}.
\end{cases}
\end{equation*}
By Lemma \ref{CL1} with $y(x,t)=\hat{q}(2\pi-x,T-t)$ and $h(x,t)=\hat{g}((2\pi-x,T-t))$, we have
\begin{equation*}
\begin{split}
\|\hat{q}\|^{2}_{L^{2}(0,T;H^{3}(\mathbb{T}))}\leq& C\|\hat{g}\|^{2}_{L^{2}(0,T;H^{1}(\mathbb{T}))}\\
\leq& Cs\int_{0}^{T}\xi^{3}e^{-4s\hat{\varphi}}\|v\|_{1}^{2}dt
+Cs\int_{0}^{T}\xi^{3}e^{-2s\hat{\varphi}}\|q\|_{1}^{2}dt\\
\leq&Cs\int_{0}^{T}\xi^{3}e^{-2s\hat{\varphi}}\|z\|_{1}^{2}dt+Cs\int_{D_{T}}\xi^{3}e^{-2s\hat{\varphi}}(q^{2}+q_{x}^{2})dxdt.
\end{split}
\end{equation*}
This implies that
\begin{equation*}
\begin{split}
s^{-1}\int_{0}^{T}\xi^{-1}\|q\|^{2}_{3}e^{-2s\hat{\varphi}}dt\leq Cs\|z\|^{2}_{L^{2}(0,T;H^{1}(\mathbb{T}))}+Cs\int_{D_{T}}\xi^{3}e^{-2s\hat{\varphi}}(q^{2}+q_{x}^{2})dxdt.
\end{split}
\end{equation*}
Combining this with \eqref{C11}, taking $s$ sufficiently large, the last term in the above inequality can be absorbed by the left hand side of \eqref{C11}. Thus we obtain that
\begin{equation*}
\int_{D_{T}}(s\xi q_{xx}^{2}+s^{3}\xi^{3}q_{x}^{2}+s^{5}\xi^{5}q^{2})e^{-2s\hat{\varphi}}dxdt
\leq Cs^{2}\|z\|^{2}_{L^{2}(0,T;H^{1}(\mathbb{T}))}+Cs^{5}\int_{D_{\omega}}\xi^{5}q^{2}e^{-6s\check{\varphi}+4s\hat{\varphi}}dxdt.
\end{equation*}
By system \eqref{C7} and Lemma \ref{CL1}, we can deduce that
\begin{equation*}
\begin{split}
&\int_{D_{T}}(s\xi v_{xx}^{2}+s^{3}\xi^{3}v_{x}^{2}+s^{5}\xi^{5}v^{2})e^{-4s\hat{\varphi}}dxdt\\
=&\int_{D_{T}}(s\xi (q_{xx}+z_{xx})^{2}+s^{3}\xi^{3}(q_{x}+z_{x})^{2}+s^{5}\xi^{5}(q+z)^{2})e^{-2s\hat{\varphi}}dxdt\\
\leq& Cs^{5}\|z\|^{2}_{L^{2}(0,T;H^{2}(\mathbb{T}))}+C\int_{D_{T}}(s\xi q_{xx}^{2}+s^{3}\xi^{3}q_{x}^{2}+s^{5}\xi^{5}q^{2})e^{-2s\hat{\varphi}}dxdt\\
\leq&Cs^{5}\|z\|^{2}_{L^{2}(0,T;H^{2}(\mathbb{T}))}+Cs^{5}\int_{D_{\omega}}\xi^{5}(e^{-s\hat{\varphi}}v-z)^{2}e^{-6s\check{\varphi}+4s\hat{\varphi}}dxdt\\
\leq&Cs^{5}\|z\|^{2}_{L^{2}(0,T;H^{2}(\mathbb{T}))}+Cs^{5}\int_{D_{\omega}}\xi^{5}v^{2}e^{-6s\check{\varphi}+2s\hat{\varphi}}dxdt\\
\leq&Cs^{5}\int_{D_{T}}g^{2}e^{-2s\hat{\varphi}}dxdt+ Cs^{5}\int_{D_{\omega}}\xi^{5}v^{2}e^{-6s\check{\varphi}+2s\hat{\varphi}}dxdt.
\end{split}
\end{equation*}
This completes the proof of Theorem \ref{CT1}.
\end{proof}

\par
Now we give the proofs of Lemma \ref{CL1} and Lemma \ref{CL2}.
\begin{proof}[Proof of Lemma \ref{CL1}]
Multiplying \eqref{C12} by $y$, integrate the result with respect to $x$ over $\mathbb{T}$, we obtain that
\begin{equation*}
\begin{split}
\frac{1}{2}\frac{d}{dt}\|y\|^{2}+\|y_{x}\|^{2}\leq& C(\|a\|_{\infty}\|y_{x}\|\|y\|+\|b\|_{\infty}\|y\|^{2}+\|y\|_{1}\|h\|_{-1})\\
\leq&\frac{1}{4}\|y_{x}\|^{2}+C\|a\|^{2}_{1}\|y\|^{2}+C\|b\|_{1}\|y\|^{2}+\frac{1}{4}\|y_{x}\|^{2}+C\|y\|^{2}+C\|h\|_{-1}^{2}\\
\leq&\frac{1}{2}\|y_{x}\|^{2}+C(\|a\|^{2}_{1}+\|b\|^{2}_{1}+1)\|y\|^{2}+C\|h\|_{-1}^{2}.
\end{split}
\end{equation*}
Applying Gronwall inequality, we can see that
\begin{equation}\label{C13}
\|y\|^{2}_{X_{0}}\leq C\|h\|^{2}_{L^{2}(0,T;H^{-1}(\mathbb{T}))}.
\end{equation}
Multiplying \eqref{C12} by $-y_{xx}$, integrate the result with respect to $x$ over $\mathbb{T}$, we obtain that
\begin{equation*}
\begin{split}
&\frac{1}{2}\frac{d}{dt}\|y_{x}\|^{2}+\|y_{xx}\|^{2}\\
\leq& C(\|a\|_{\infty}\|y_{x}\|\|y_{xx}\|+\|b_{x}\|\|y\|_{\infty}\|y_{x}\|+\|b\|_{\infty}\|y_{x}\|^{2}+\|y_{xx}\|\|h\|)\\
\leq&\frac{1}{4}\|y_{xx}\|^{2}+C\|a\|^{2}_{1}\|y_{x}\|^{2}+C\|b\|_{1}\|y_{x}\|^{2}+C\|b\|_{1}\|y\|\|y_{x}\|+\frac{1}{4}\|y_{xx}\|^{2}+C\|h\|^{2}\\
\leq&\frac{1}{2}\|y_{xx}\|^{2}+C(\|a\|^{2}_{1}+\|b\|^{2}_{1}+1)\|y_{x}\|^{2}+C\|y\|^{2}+C\|h\|^{2}.
\end{split}
\end{equation*}
Applying \eqref{C13} and Gronwall inequality, we can see that
\begin{equation}\label{C14}
\|y\|^{2}_{X_{1}}\leq C\|h\|^{2}_{L^{2}(0,T;L^{2}(\mathbb{T}))}.
\end{equation}
Multiplying \eqref{C12} by $y_{xxxx}$, integrate the result with respect to $x$ over $\mathbb{T}$, we obtain that
\begin{equation*}
\begin{split}
&\frac{1}{2}\frac{d}{dt}\|y_{xx}\|^{2}+\|y_{xxx}\|^{2}\\
\leq& C(\|a_{x}\|\|y_{x}\|_{\infty}\|y_{xxx}\|+\|a\|_{\infty}\|y_{xx}\|^{2}+\|b_{x}\|\|y\|_{\infty}\|y_{xxx}\|+\|b\|_{\infty}\|y_{x}\|\|y_{xxx}\|+\|y_{xxx}\|\|h_{x}\|)\\
\leq&\frac{1}{4}\|y_{xxx}\|^{2}+C\|a\|^{2}_{1}\|y\|_{2}^{2}+C\|a\|_{1}\|y_{xx}\|^{2}+\frac{1}{4}\|y_{xxx}\|^{2}+C\|b\|_{1}^{2}\|y\|_{1}^{2}+\frac{1}{4}\|y_{xxx}\|^{2}+C\|h\|_{1}^{2}\\
\leq&\frac{3}{4}\|y_{xxx}\|^{2}+C(\|a\|^{2}_{1}+1)\|y_{xx}\|^{2}+C(\|a\|^{2}_{1}+\|b\|^{2}_{1})\|y\|_{1}^{2}+C\|h\|_{1}^{2}.
\end{split}
\end{equation*}
Applying \eqref{C14} and Gronwall inequality, we can see that
\begin{equation*}
\|y\|^{2}_{X_{2}}\leq C\|h\|^{2}_{L^{2}(0,T;H^{1}(\mathbb{T}))}.
\end{equation*}
Thus the conclusion of lemma \ref{CL1} follows.
\end{proof}

\begin{proof}[Proof of Lemma \ref{C2}]
Noting that $w=e^{-s \varphi}q$, simple calculations give that
\begin{equation*}
e^{-s\varphi}Lq=e^{-s\varphi}L(e^{s\varphi}w):=I_{1}+I_{2}+I_{3},
\end{equation*}
where
\begin{equation*}
\begin{split}
&I_{1}=w_{t}+w_{xxx}+3s^{2}\varphi_{x}^{2}w_{x}+3s^{2}\varphi_{x}\varphi_{xx}w,\\
&I_{2}=3s\varphi_{x}w_{xx}+3s\varphi_{xx}w_{x}+s^{3}\varphi_{x}^{3}w,\\
&I_{3}=s(\varphi_{t}+\varphi_{xxx})w.
\end{split}
\end{equation*}
It is not difficult to deduce that
\begin{equation}\label{C3}
2\int_{D_{T}}I_{1}I_{2}dxdt\leq \|I_{1}+I_{2}\|^{2}_{L^{2}(D_{T})}=\|e^{-s\varphi}Lq-I_{3}\|^{2}_{L^{2}(D_{T})}\leq 2\|e^{-s\varphi}Lq\|^{2}_{L^{2}(D_{T})}+2\|I_{3}\|^{2}_{L^{2}(D_{T})}.
\end{equation}
To calculate each term of $2\int_{D_{T}}I_{1}I_{2}dxdt$, let $I_{ij}~(i-1,2,~j=1,2,3,4)$ denote the $j$th term in the expression of $I_{i}$. Then, performing integrations by parts with respect to $x,t$, we can obtain that
\begin{align*}
&2\int_{D_{T}}I_{11}I_{21}dxdt=6s\int_{D_{T}}\varphi_{x}w_{t}w_{xx}dxdt=-6s\int_{D_{T}}\varphi_{xx}w_{t}w_{x}dxdt+3s\int_{D_{T}}\varphi_{xt}w_{x}^{2}dxdt,\\
&2\int_{D_{T}}I_{11}I_{22}dxdt=6s\int_{D_{T}}\varphi_{xx}w_{t}w_{x}dxdt,\\
&2\int_{D_{T}}I_{11}I_{23}dxdt=2s^{3}\int_{D_{T}}\varphi^{3}_{x}w_{t}wdxdt=-3s^{3}\int_{D_{T}}\varphi_{x}^{2}\varphi_{xt}w^{2}dxdt,\\
&2\int_{D_{T}}I_{12}I_{21}dxdt=6s\int_{D_{T}}\varphi_{x}w_{xxx}w_{xx}dxdt=-3s\int_{D_{T}}\varphi_{xx}w_{xx}^{2}dxdt,\\
&2\int_{D_{T}}I_{12}I_{22}dxdt=6s\int_{D_{T}}\varphi_{xx}w_{xxx}w_{x}dxdt=3s\int_{D_{T}}\varphi_{xxxx}w^{2}_{x}dxdt-6s\int_{D_{T}}\varphi_{xx}w_{xx}^{2}dxdt,\\
&2\int_{D_{T}}I_{12}I_{23}dxdt=2s^{3}\int_{D_{T}}\varphi_{x}^{3}w_{xxx}wdxdt=-3s^{3}\int_{D_{T}}(\varphi_{x}^{2}\varphi_{xx})_{xx}w^{2}dxdt+9s^{3}\int_{D_{T}}\varphi_{x}^{2}\varphi_{xx}w_{x}^{2}dxdt,\\
&2\int_{D_{T}}I_{13}I_{21}dxdt=18s^{3}\int_{D_{T}}\varphi_{x}^{3}w_{x}w_{xx}dxdt=-27s^{3}\int_{D_{T}}\varphi_{x}^{2}\varphi_{xx}w_{x}^{2}dxdt,\\
&2\int_{D_{T}}I_{13}I_{22}dxdt=18s^{3}\int_{D_{T}}\varphi_{x}^{2}\varphi_{xx}w_{x}^{2}dxdt,\\
&2\int_{D_{T}}I_{13}I_{23}dxdt=6s^{5}\int_{D_{T}}\varphi_{x}^{5}w_{x}wdxdt=-15s^{5}\int_{D_{T}}\varphi_{x}^{4}\varphi_{xx}w^{2}dxdt,\\
&2\int_{D_{T}}I_{14}I_{21}dxdt=18s^{3}\int_{D_{T}}\varphi_{x}^{2}\varphi_{xx}ww_{xx}dxdt=9s^{3}\int_{D_{T}}(\varphi_{x}^{2}\varphi_{xx})_{xx}w^{2}dxdt
                           -18s^{3}\int_{D_{T}}\varphi_{x}^{2}\varphi_{xx}w_{x}^{2}dxdt,\\
&2\int_{D_{T}}I_{14}I_{22}dxdt=18s^{3}\int_{D_{T}}\varphi_{x}\varphi_{xx}^{2}ww_{x}dxdt=-9s^{3}\int_{D_{T}}(\varphi_{x}\varphi_{xx}^{2})_{x}w^{2}dxdt,\\
&2\int_{D_{T}}I_{14}I_{23}dxdt=6s^{5}\int_{D_{T}}\varphi_{x}^{4}\varphi_{xx}w^{2}dxdt.
\end{align*}
Gathering together all the computations, we have
\begin{equation}\label{C4}
2\int_{D_{T}}I_{1}I_{2}dxdt=\int_{D_{T}}(-9s\varphi_{xx}w_{xx}^{2}-18s^{3}\varphi_{x}^{2}\varphi_{xx}w_{x}^{2}-9s^{5}\varphi_{x}^{4}\varphi_{xx}w^{2})dxdt+R,
\end{equation}
where
\begin{equation*}
\begin{split}
|R|=&\Big|\int_{D_{T}}\Big[3s\left(\varphi_{xt}+\varphi_{xxxx}\right)w_{x}^{2}
+s^{3}\left(-3\varphi_{x}^{2}\varphi_{xt}+6(\varphi_{x}^{2}\varphi_{xx})_{xx}-9(\varphi_{x}\varphi_{xx}^{2})_{x}\right)w^{2}\Big]dxdt\Big|\\
\leq& C\int_{D_{T}}(s^{2}\varphi^{2}w_{x}^{2}+s^{4}\varphi^{4}w^{2})dxdt.
\end{split}
\end{equation*}
According to \eqref{C2}, we can obtain that
\begin{equation}\label{C5}
\begin{split}
&\int_{D_{T}}(-9s\varphi_{xx}w_{xx}^{2}-18s^{3}\varphi_{x}^{2}\varphi_{xx}w_{x}^{2}-9s^{5}\varphi_{x}^{4}\varphi_{xx}w^{2})dxdt\\
\geq&C\int_{D_{T}}(s\varphi w_{xx}^{2}+s^{3}\varphi^{3}w_{x}^{2}+s^{5}\varphi^{5}w^{2})dxdt-C\int_{D_{\omega}}(s\varphi w_{xx}^{2}+s^{3}\varphi^{3}w_{x}^{2}+s^{5}\varphi^{5}w^{2})dxdt.
\end{split}
\end{equation}
Combining \eqref{C3}-\eqref{C5}, we have
\begin{equation*}
\begin{split}
&\int_{D_{T}}(s\varphi w_{xx}^{2}+s^{3}\varphi^{3}w_{x}^{2}+s^{5}\varphi^{5}w^{2})dxdt\\
\leq&C\|e^{-s\varphi}Lq\|^{2}_{L^{2}(D_{T})}+C\|I_{3}\|^{2}_{L^{2}(D_{T})}+C\int_{D_{T}}(s^{2}\varphi^{2}w_{x}^{2}+s^{4}\varphi^{4}w^{2})dxdt\\
&+C\int_{D_{\omega}}(s\varphi w_{xx}^{2}+s^{3}\varphi^{3}w_{x}^{2}+s^{5}\varphi^{5}w^{2})dxdt\\
\leq&C\int_{D_{T}}|Lq|^{2}e^{-2s\varphi}dxdt+C\int_{D_{T}}(s^{2}\varphi^{2}w_{x}^{2}+s^{4}\varphi^{4}w^{2})dxdt\\
&+C\int_{D_{\omega}}(s\varphi w_{xx}^{2}+s^{3}\varphi^{3}w_{x}^{2}+s^{5}\varphi^{5}w^{2})dxdt.
\end{split}
\end{equation*}
Choosing $s_{1}$ large enough, for any $s\geq s_{1}$, the second term in the right hand side can absorbed by the left hand side, this means that
\begin{equation*}
\begin{split}
&\int_{D_{T}}(s\varphi w_{xx}^{2}+s^{3}\varphi^{3}w_{x}^{2}+s^{5}\varphi^{5}w^{2})dxdt\\
\leq&C\int_{D_{T}}|Lq|^{2}e^{-2s\varphi}dxdt+C\int_{D_{\omega}}(s\varphi w_{xx}^{2}+s^{3}\varphi^{3}w_{x}^{2}+s^{5}\varphi^{5}w^{2})dxdt.
\end{split}
\end{equation*}
\par
This completes the proof of Lemma \ref{CL2}.
\end{proof}

\subsection{Observability inequality for KdVB equation}
\par
Now, we build observability inequality for \eqref{C1} with $g\equiv0$.
\begin{proposition}\label{CP1}
Let $a,b\in X_{0}$ and $g\equiv0$, there exists constant $C=C(a,b,T,\omega)>0$ such that for any $v_{T}\in L^{2}(\mathbb{T})$, the corresponding solution to \eqref{C1} satisfies:
\begin{equation*}
\|v(0)\|^{2}\leq C\int_{D_{\omega}}v^{2}dxdt.
\end{equation*}
\end{proposition}
\begin{proof}[Proof of Proposition \ref{CP1}]
By a similar proof as in Lemma \ref{CL1}, we obtain that
\begin{equation}\label{C15}
\|v(t)\|^{2}\leq C\|v_{T}\|^{2}~~\forall~t\in[0,T].
\end{equation}
Replacing $v(t)$ by $v(0)$ and $v_{T}$ by $v(\tau)$ for $T/3<\tau<2T/3$ in \eqref{C15}, and integrating over $\tau\in(T/3,2T/3)$, we arrive at
\begin{equation}\label{C16}
\|v(0)\|^{2}\leq C\int_{\frac{T}{3}}^{\frac{2T}{3}}\|v(\tau)\|^{2}d\tau.
\end{equation}
\par
It follows from Theorem \ref{CT1} that
\begin{equation*}
\int_{D_{T}}\xi^{5}v^{2}e^{-4s\hat{\varphi}}dxdt\leq \int_{D_{\omega}}\xi^{5}e^{-6s\check{\varphi}+2s\hat{\varphi}}v^{2}dxdt.
\end{equation*}
This implies
\begin{equation*}
\int_{\frac{T}{3}}^{\frac{2T}{3}}\|v(\tau)\|^{2}d\tau\leq C\int_{D_{\omega}}v^{2}dxdt.
\end{equation*}
Combining \eqref{C16},the conclusion of Proposition \ref{CP1} follows.
\par
This completes the proof.
\end{proof}

\subsection{Truncated observability inequality for KdVB equation}
\begin{theorem}\label{pro4}
For any $\rho>0$ and any integer $N\geq1,$ there is an integer $M=M(N)\geq1$ such that if $a,b\in B_{X_{0}}(\rho),$ then any solution $v\in X_{T}$ of \eqref{C1} corresponding to $v_{T}\in H_{N}$ satisfies
\begin{equation}\label{10}
\begin{split}
\|v(0)\|\leq C(\delta,\rho,N,T)\|\Pi_{M}(\chi v)\|_{L^{2}(D_{T})}.
\end{split}
\end{equation}
\end{theorem}

\begin{proof}[Proof of Theorem \ref{pro4}]
We divide the proof into several steps.
\par
\textbf{Step 1.} We claim that if $v\in X_{T}$ is a solution of \eqref{C1} with $v_{T}\in H_{N},$ then
\begin{equation}\label{11}
\begin{split}
\|\chi v\|_{H^{1}(D_{T})}\leq C(\delta,\rho,N,T)\|\chi v\|_{L^{2}(D_{T})}.
\end{split}
\end{equation}
\par
Indeed, if we suppose that \eqref{11} is false, then there are functions $a_{n},b_{n}\in B_{X_{0}}(\rho)$ and solutions $v_{n}\in X_{0}$
of \eqref{C1} with $a=a_{n},b=b_{n}$ such that
\begin{equation}\label{16}
\begin{split}
&v_{n}(T)\in H_{N},~\|v_{n}(T)\|=1,\\
&\|\chi v_{n}\|_{H^{1}(D_{T})}\geq n\|\chi v_{n}\|_{L^{2}(D_{T})}.
\end{split}
\end{equation}
According to the first line of \eqref{16}, we have $\|v_{n}(T)\|_{2}\leq C$ for all $n\geq1,$ by standard estimates for \eqref{C1},
it holds that
\begin{equation*}
\begin{split}
\|v_{n}\|_{X_{2}}+\|\partial_{t}v_{n}\|_{L^{2}(0,T;L^{2}(\mathbb{T}))}\leq C,~~n\geq1.
\end{split}
\end{equation*}
Passing to a subsequence, we can assume that
\begin{equation*}
\begin{split}
&v_{n}(T)\rightarrow \hat{v}(T)~~{\rm{in}}~~H_{N},
\\&\partial_{t}v_{n}\rightarrow \partial_{t}\hat{v}~~{\rm{weakly~~ in}}~~L^{2}(0,T;L^{2}(\mathbb{T})),
\\&v_{n}\rightarrow \hat{v}~~{\rm{strongly~~in}}~~L^{2}(0,T;H^{s}(\mathbb{T})),
\\&a_{n}\rightarrow \hat{a}~~{\rm{weakly~~in}}~~L^{2}(0,T;H^{1}(\mathbb{T})),
\\&b_{n}\rightarrow \hat{b}~~{\rm{weakly~~in}}~~L^{2}(0,T;H^{1}(\mathbb{T})),
\end{split}
\end{equation*}
where $s\in (0,3)$, by passing to the limit, we have $\hat{v}\in X_{T}$ is the solution of \eqref{C1} with $a=\hat{a},b=\hat{b},v_{T}=\hat{v}(T)$ and $\|\hat{v}(T)\|=1.$ Moreover, we have
$\chi \hat{v}\equiv0$ in $D_{T}.$ Let an interval $(t_{1},t_{2})\subset (0,T)$ and an interval $\omega \subset I$ be such that $\chi(t,x)\geq \alpha>0$ for $(t,x)\in (t_{1},t_{2})\times \omega.$
By applying Proposition \ref{CP1} to $\hat{v},$ we have
\begin{equation*}
\begin{split}
\|\hat{v}(t_{1})\|\leq C\|\hat{v}\|_{L^{2}((t_{1},t_{2})\times\omega)}
\leq C\alpha^{-1}\|\chi\hat{v}\|_{L^{2}((t_{1},t_{2})\times\omega)}
\leq C\alpha^{-1}\|\chi\hat{v}\|_{L^{2}(D_{T})}=0,
\end{split}
\end{equation*}
due to the backward uniqueness for \eqref{C1}, it holds that $\hat{v}(t)=0$ for $t_{1}\leq t\leq T,$ this contradicts the fact that $\|\hat{v}(T)\|=1.$
\par
\textbf{Step 2.} We prove that
\begin{equation}\label{5}
\begin{split}
\|\chi v\|_{L^{2}(D_{T})}\leq C\|\Pi_{M}(\chi v)\|_{L^{2}(D_{T})}.
\end{split}
\end{equation}
\par
Indeed,
\begin{equation*}
\begin{split}
\|\chi v\|^{2}_{L^{2}(D_{T})}
&\leq \|\Pi_{M}(\chi v)\|^{2}_{L^{2}(D_{T})}+\|(I-\Pi_{M})(\chi v)\|^{2}_{L^{2}(D_{T})}\\
&\leq \|\Pi_{M}(\chi v)\|^{2}_{L^{2}(D_{T})}+\alpha_{M}^{-1}\|(I-\Pi_{M})(\chi v)\|^{2}_{H^{1}(D_{T})}\\
&\leq \|\Pi_{M}(\chi v)\|^{2}_{L^{2}(D_{T})}+C\alpha_{M}^{-1}\|\chi v\|^{2}_{H^{1}(D_{T})}\\
&\leq \|\Pi_{M}(\chi v)\|^{2}_{L^{2}(D_{T})}+C\alpha_{M}^{-1}\|\chi v\|^{2}_{L^{2}(D_{T})},
\end{split}
\end{equation*}
where we have used \eqref{11} in the last inequality. By taking $M$ large enough such that $C\alpha_{M}^{-1}\leq \frac{1}{2},$ this implies that \eqref{5} holds.

\par
\textbf{Step 3.} We prove \eqref{10}.
\par
Since $\omega_{\chi}=\{(x,t)\in D_{T}~|~|\chi(x,t)|>\rho\}$ is nonempty for a
sufficiently small $\rho> 0,$ there exists a open set $\omega_{0}\subset I$ such that $\omega_{0}\times(\alpha,\beta)\subset \omega_{\chi},$ it follows from Proposition \ref{CP1} that
\begin{equation*}
\begin{split}
\|v(0)\|^{2}
\leq C\|v(\alpha)\|^{2}
\leq C\int_{\omega_{0}\times(\alpha,\beta)}v^{2}dxdt
\leq C\int_{\omega_{\chi}}v^{2}dxdt
\leq C(\chi,\rho)\int_{0}^{T}\|\chi v\|^{2}dt.
\end{split}
\end{equation*}
Combining this inequality with \eqref{5}, we can obtain \eqref{10}.
\par
This completes the proof of Theorem \ref{pro4}.
\end{proof}

\section{Proof of Theorem \ref{MT}}
\subsection{Control method}
For random dynamical system in discrete time, we have the following result.
Let $H$ be a separable Hilbert space, and $E$ be a separable Banach space. We consider a random dynamical system of the form
\begin{equation}\label{SC1}
\begin{split}
u_{k}=S(u_{k-1},\eta_{k}),~~k\geq1,
\end{split}
\end{equation}
where $S:H\times E\rightarrow H$ is a continuous mapping and $\eta_{k}$ is a sequence of
i.i.d. random variables in $E.$ Let $\mathcal{K}\subset E$ be the support of the law $\ell:=\mathcal{D}(\eta_{k})$.
Let us consider random dynamical system \eqref{SC1} in a \textbf{compact} metric space $(X,d)$
and $S:X\times E\rightarrow X$ is a continuous mapping. Equation \eqref{SC1} is supplemented with the initial
condition
\begin{equation}\label{SC5}
\begin{split}
u_{0}=u,
\end{split}
\end{equation}
where $u$ is an $X-$valued random variable independent of $\eta_{k}.$ We denote by $(u_{k},\mathbb{P}_{u})$
the discrete-time Markov process associated with \eqref{SC1} and by $P_{k}(u,\Gamma)$ its transition function.
The Markov operators corresponding to $P_{k}(u,\Gamma)$ are denoted by
\begin{equation*}
\begin{split}
\mathcal{B}_{k}:C(X)\rightarrow C(X),~~~\mathcal{B}_{k}^{*}:\mathcal{P}(X)\rightarrow\mathcal{P}(X),k\geq0.
\end{split}
\end{equation*}
\par
The following conditions are assumed to be satisfied for the mapping $S$ and the measure $\ell$.
\par
\textit{\textbf{(R)~Regularity}.} Suppose that $S:H\times E\rightarrow H$ is a $C^{1}$-smooth mapping such that $S(X\times \mathcal{K})\subset X$.

\par
\textit{\textbf{(ACP)~Approximate controllability to a given point}.} Let $\hat{u}\in X$ be a point and let $\mathcal{K}\subset X$ be a
compact subset. System \eqref{SC1} is said to be globally approximately controllable to $\hat{u}$ by
a $\mathcal{K}$-valued control if for any $\varepsilon>0$ there exists $m\geq1$ such that, given any initial point
$u\in X,$ we can find $\zeta^{u}_{1},\cdots,\zeta^{u}_{m}\in\mathcal{K}$ for which
\begin{equation}\label{SC2}
\begin{split}
d(S_{m}(u;\zeta^{u}_{1},\cdots,\zeta^{u}_{m}),\hat{u})\leq\varepsilon,
\end{split}
\end{equation}
where $S_{k}(u;\eta_{1},\cdots,\eta_{k})$ denotes the trajectory of \eqref{SC1}, \eqref{SC5}.
\par
\textit{\textbf{\textbf{(LS)~Local stabilisability}}.} Let us set $\mathbf{B}_{\delta}=\{(u,u^{\prime})\in X\times X~:~d(u,u^{\prime})\leq \delta\}.$ We say that \eqref{SC1}
is locally stabilisable if for any $R>0$ and any compact set $\mathcal{K}\subset E,$ there is a finite dimensional
subspace $\mathcal{E}\subset E,$ positive numbers $C,\delta,\alpha\leq1,$ and $q < 1,$ and a continuous
mapping
$$\Phi:\mathbf{B}_{\delta}\times B_{E}(R)\rightarrow\mathcal{E},~(u,u^{\prime},\eta)\rightarrow \eta^{\prime},$$
which is continuously differentiable in $\eta$ and satisfies the following inequalities for any
$(u,u^{\prime})\in\mathbf{B}_{\delta}:$
\begin{equation}\label{SC3}
\begin{split}
&\sup\limits_{\eta\in B_{E}(R)}(\|\Phi(u,u^{\prime},\eta)\|_{E}+\|D_{\eta}\Phi(u,u^{\prime},\eta)\|_{\mathcal{L}(E)})\leq Cd(u,u^{\prime})^{\alpha},\\
&\sup\limits_{\eta\in \mathcal{K}}d(S(u,\eta),S(u^{\prime},\eta+\Phi(u,u^{\prime},\eta)))\leq qd(u,u^{\prime}).
\end{split}
\end{equation}
\par
\textit{\textbf{(D) Decomposable of the noise}.} For the random variables $\eta_{k},$ we assume that their law $\ell$ is \textbf{decomposable} in the following sense. There are two
sequences $\{F_{n}\}$ and $\{G_{n}\}$ of closed subspaces in $E$ possessing the two properties below:
\par
(a) $\dim F_{n}<+\infty$ and $F_{n}\subset F_{n+1}$ for any $n\geq1,$ and the vector space $\bigcup_{n}F_{n}$ is dense in $E.$
\par
(b) $E$ is the direct sum of $F_{n}$ and $G_{n}$, the norms of the corresponding projections $P_{n}$
and $Q_{n}$ are bounded uniformly in $n\geq1,$ and the measure $\ell$ can be written as the
product of its projections $P_{n\ast}\ell$ and $Q_{n\ast}\ell$ for any $n\geq1.$ Moreover, $P_{n\ast}\ell$ possess $C^{1}$-smooth densities
$\rho_{n}$ with respect to the Lebesgue measure on $F_n.$
\par
~~
\par
We assume that the phase space $X$ is a compact subset of a Banach space $H,$ endowed with a norm $\|\cdot\|.$
\begin{proposition}\label{proSC}\cite[Theorem 1.1]{S2}
Assume that Hypotheses (R), (ACP), (LS), and (D) are satisfied, then \eqref{SC1} has a unique stationary measure $\mu\in \mathcal{P}(X),$ and $\mu$ is exponentially mixing in the dual-Lipschitz metric, namely, there are positive numbers $\gamma$ and $C$ such that
\begin{equation*}
\begin{split}
\|\mathcal{B}_{k}^{*}\lambda-\mu\|_{L}^{*}\leq Ce^{-\gamma k}~{\rm{for}}~k\geq0,\lambda\in \mathcal{P}(X).
\end{split}
\end{equation*}

\end{proposition}

\begin{remark}
Proposition \ref{proSC} has been used to establish exponential mixing for 2D Navier-Stokes equations driven by a space-time localised noise and boundary noise in \cite{S1} and \cite{S2}, respectively.
\end{remark}

\subsection{Squeezing property}
\par
We first introduce a minimisation problem:
\par
\textbf{(P)} Given a constant $\delta>0,$ an integer $N\geq1,$ and functions $u_{0}\in H$ and
$\hat{u}\in X_{0}$ satisfying \eqref{1} with $\eta=0$, minimise the functional
$$J(w,\zeta)=\frac{1}{2}\int_{0}^{T}\|\zeta(t)\|^{2}dt+\frac{1}{\delta}\|P_{N}u(T)\|^{2}$$
over the set of functions $(w,\zeta)\in X_{0}\times L^{2}(D_{T})$ satisfying the system
\begin{eqnarray}
\begin{array}{l}
\left\{
\begin{array}{lll}
w_{t}+Aw+\hat{u}_{x}w+\hat{u}w_{x}=\chi(\Pi_{m}\zeta)\\
w(x,0)=u_{0}
\end{array}
\right.
\end{array}
\begin{array}{lll}\textrm{in}~D_{T},
\\\textrm{in}~\mathbb{T}.
\end{array}
\end{eqnarray}

\begin{lemma}\label{pro5} The following results hold.
\par
(1) Problem \textbf{(P)} has a unique optimal solution $(w,\zeta),$ if we denote by $\Psi(\hat{u})(v_0):=\zeta,$ then, the mapping $\Psi$
is a mapping from $X_{0}$ to $\mathcal{L}(H;L^{2}(D_{T}))$, and it is infinitely differentiable and uniformly Lipschitz continuous on balls.
\par
(2) The optimal solution $(w,\zeta)$ satisfies the inequality
\begin{equation}\label{27}
\begin{split}
\frac{1}{\delta}\|P_{N}w(T)\|^{2}+\|\zeta\|^{2}_{L^{2}(D_{T})}\leq C\|v_{0}\|^{2},
\end{split}
\end{equation}
where $C>0$ is a constant not depending on $N$ and $\delta.$
\end{lemma}
\begin{proof}
By the same arguments as in Step 1 and Step 2 in \cite[Section 3.2]{S1}, we can obtain (1), thus, we omit its proof.
For (2), by the same arguments as in Step 3 in \cite[Section 3.2]{S1}, it holds that
\begin{equation*}
\begin{split}
\frac{2}{\delta}\|P_{N}w(T)\|^{2}+\|\zeta\|^{2}_{L^{2}(D_{T})}=-(\theta(0),v_{0}),
\end{split}
\end{equation*}
where $\zeta=\Pi_{m}(\chi \theta)$ and $\theta$ is a solution to the system
\begin{eqnarray*}
\begin{array}{l}
\left\{
\begin{array}{lll}
-\theta_{t}+A\theta+\hat{u}\theta_{x}=0\\
\theta(x,T)=-\frac{2}{\delta}P_{N}w(T)
\end{array}
\right.
\end{array}
\begin{array}{lll}\textrm{in}~D_{T},
\\\textrm{in}~\mathbb{T}.
\end{array}
\end{eqnarray*}
It is easy to see that
\begin{equation*}
\begin{split}
\frac{2}{\delta}\|P_{N}w(T)\|^{2}+\|\zeta\|^{2}_{L^{2}(D_{T})}&\leq |(\theta(0),v_{0})|
\leq C\|\Pi_{m}(\chi \theta)\|_{L^{2}(D_{T})}\|v_{0}\|=C\|\zeta\|_{L^{2}(D_{T})}\|v_{0}\|,
\end{split}
\end{equation*}
where we have used the truncated observability inequality \eqref{10} in Theorem \ref{pro4}, this implies that
\begin{equation*}
\begin{split}
\|\zeta\|_{L^{2}(D_{T})}\leq C\|v_{0}\|,
\end{split}
\end{equation*}
according to this estimate, we have \eqref{27}.
\par
This completes the proof.
\end{proof}

\par
Now, we can establish the squeezing property.
Let $u$ be the solution of \eqref{1} with $\eta$ in \eqref{39} and $\hat{u}$ be the solution of the system
\begin{eqnarray}\label{26}
 \begin{array}{l}
 \left\{
 \begin{array}{llll}
 \hat{u}_{t}+A\hat{u}+B(\hat{u})=h
 \\\hat{u}(x,0)=\hat{u}_{0}(x)
 \end{array}
 \right.
 \end{array}
 \begin{array}{lll}{\rm{in}}~D_{T},
\\{\rm{in}}~\mathbb{T}.
\end{array}
\end{eqnarray}
Set $v:=u-\hat{u},$ it is easy to see that $v$ satisfies
\begin{eqnarray}\label{4}
\begin{array}{l}
\left\{
\begin{array}{lll}
v_{t}+Av +\hat{u}_{x}v+\hat{u}v_{x}+B(v)=\eta\\
v(x,0)=v_{0}:=u_{0}-\hat{u}_{0}
\end{array}
\right.
\end{array}
\begin{array}{lll}\textrm{in}~D_{T},
\\\textrm{in}~\mathbb{T}.
\end{array}
\end{eqnarray}

\begin{proposition}\label{pro6}
Under the above hypotheses, for any $R > 0$ and $q\in(0,1),$ there is an
integer $m\geq1,$ positive constants $d$ and $C,$ and a continuous mapping
\begin{equation*}
\begin{split}
\Upsilon:B_{H^{2}(D_{T})}(R)\times B_{H}(R)\rightarrow \mathcal{L}(H,\mathcal{E}_{m})\\
(h,\hat{u}_{0})~~~~~~~\longmapsto \Upsilon(h,\hat{u}_{0}),
\end{split}
\end{equation*}
such that the following properties hold.
\par
\textbf{Contraction}: For any functions $h\in B_{H^{2}(D_{T})}(R)$ and $u_{0},\hat{u}_{0}\in B_{H}(R)$ satisfying the inequality $\|u_{0}-\hat{u}_{0}\|\leq d,$ we have
\begin{equation}\label{21}
\begin{split}
\|S(\hat{u}_{0},h)-S(u_{0},h+\Upsilon(h,\hat{u}_{0})(u_{0}-\hat{u}_{0}))\|\leq q\|u_{0}-\hat{u}_{0}\|.
\end{split}
\end{equation}
\par
\textbf{Regularity}: The mapping $\Upsilon$ is infinitely smooth in the Fr\'{e}chet sense.
\par
\textbf{Lipschitz continuity}: The mapping $\Upsilon$ is Lipschitz continuous with the constant $C.$ That is,
$$\|\Upsilon(h_{1},\hat{u}_{1})-\Psi(h_{2},\hat{u}_{2})\|_{\mathcal{L}}\leq C(\|h_{1}-h_{2}\|_{2}+\|\hat{u}_{1}-\hat{u}_{2}\|),$$
where $\|\cdot\|_{\mathcal{L}}$ stands for the norm in the space $\mathcal{L}(H,\mathcal{E}_{m}).$
\end{proposition}
\begin{proof}[Proof of Proposition \ref{pro6}]
\par
We define an operator $\mathcal{R}^{\hat{u}}$ by $\mathcal{R}^{\hat{u}}(v_{0},\eta):=w,$ where $w$ is the solution of system
\begin{eqnarray}\label{3}
\begin{array}{l}
\left\{
\begin{array}{lll}
w_{t}+Aw +\hat{u}_{x}w+\hat{u}w_{x}=\eta\\
w(x,0)=v_{0}
\end{array}
\right.
\end{array}
\begin{array}{lll}\textrm{in}~D_{T},
\\\textrm{in}~\mathbb{T},
\end{array}
\end{eqnarray}
and by $\mathcal{R}^{\hat{u}}_{t}$ its restriction to the time $t.$
We divide the proof of Proposition \ref{pro6} into several steps.
\par
\textbf{Step 1.} Definition of $\Upsilon.$
\par
Indeed, it follows from Lemma \ref{pro5} that
\begin{equation}\label{17}
\begin{split}
&\|P_{N}\mathcal{R}^{\hat{u}}_{T}(v_{0},\chi \Pi_{m}(\Psi(\hat{u})v_{0}))\|\leq C\delta\|v_{0}\|,\\
&\|\Psi(\hat{u})v_{0}\|_{L^{2}(D_{T})}\leq C\|v_{0}\|.
\end{split}
\end{equation}
We define $$\Upsilon(h,\hat{u}_{0})v_0:=\chi \Pi_{m}(\Psi(\hat{u})v_{0}),$$
thus $\Upsilon(h,\hat{u}_{0}):H\rightarrow \mathcal{E}_{m}$ is a continuous linear operator, namely, $\Upsilon(h,\hat{u}_{0})\in\mathcal{L}(H,\mathcal{E}_{m})$. We will show that it satisfies the properties in Proposition \ref{pro6}.
\par
\textbf{Step 2.} Let $w$ be the solution of \eqref{3} with $\eta=\Upsilon(h,\hat{u}_{0})v_0,$ we claim that
\begin{equation}\label{18}
\begin{split}
\|w(T)\|\leq \frac{q}{2}\|v_{0}\|,~~\|w\|_{X_{T}}\leq C\|v_{0}\|.
\end{split}
\end{equation}
\par
Indeed, due to the first inequality in \eqref{17}, we have $\|P_{N}w(T)\|\leq C\delta\|v_{0}\|,$ this implies that
\begin{equation*}
\begin{split}
\|w(T)\|&\leq \|(1-P_{N})w(T)\|+\|P_{N}w(T)\|\\
&\leq \lambda_{N+1}^{-\frac{1}{2}}\|w(T)\|_{1}+C\delta\|v_{0}\|\\
&\leq C\lambda_{N+1}^{-\frac{1}{2}}(\|v_{0}\|+\|\eta\|_{L^{2}(D_{T})})+C\delta\|v_{0}\|\\
&\leq C(\lambda_{N+1}^{-\frac{1}{2}}+\delta)\|v_{0}\|.
\end{split}
\end{equation*}
Choosing $N$ sufficiently large and $\delta$ sufficiently small, we obtain the first inequality in \eqref{18}.
With the help of the continuity of $\mathcal{R}^{\hat{u}}$ and the boundedness of $\Upsilon$, we can derive the second inequality in \eqref{18}.
\par
\textbf{Step 3.} Let $v$ be the solution of \eqref{4} with $\eta=\Upsilon(h,\hat{u}_{0})(u_{0}-\hat{u}_{0}),$ we claim that
\begin{equation}\label{19}
\begin{split}
\|v(T)\|\leq q\|v_{0}\|.
\end{split}
\end{equation}
\par
Indeed, we rewrite $v$ in the form $v=w+z,$ where $w$ is a solution of \eqref{3} with $\eta=\Upsilon(h,\hat{u}
_{0})(u_{0}-\hat{u}_{0}),$ then $z$ must be a solution of
\begin{eqnarray}\label{20}
\begin{array}{l}
\left\{
\begin{array}{lll}
z_{t}+Az+[(\hat{u}+w)z]_{x}+B(z)=-B(w)\\
z(x,0)=0
\end{array}
\right.
\end{array}
\begin{array}{lll}\textrm{in}~D_{T},
\\\textrm{in}~\mathbb{T}.
\end{array}
\end{eqnarray}
Multiplying the first equation in \eqref{20} by $2z$ and then performing integration by parts over $I,$ we get
\begin{equation*}
\begin{split}
\frac{1}{2}\frac{d}{dt}\|z\|^{2}+\|z(t)\|_{1}^{2}
&=\int_{I}(\hat{u}+w)zz_{x}dx-\int_{I}ww_{x}zdx\\
&\leq C\|\hat{u}+w\|_{1}\|z\|\|z\|_{1}+C\|w\|\|w\|_{1}\|z\|_{1}.
\end{split}
\end{equation*}
With the help of Young's inequality and Poincar\'{e} inequality, we derive
\begin{equation*}
\begin{split}
\frac{d}{dt}\|z\|^{2}+\|z\|_{1}^{2}
\leq C(\|\hat{u}+w\|_{1}^{2}\|z\|^{2}+\|w\|^{2}\|w\|_{1}^{2})
.
\end{split}
\end{equation*}
By using Gronwall's inequality, it holds that
\begin{equation*}
\begin{split}
\|z(t)\|^{2}\leq C\int_{0}^{t}\|w(s)\|^{2}\|w(s)\|_{1}^{2}ds\cdot e^{C\int_{0}^{t}\|\hat{u}(s)+w(s)\|_{1}^{2}ds}
,
\end{split}
\end{equation*}
this implies that
\begin{equation*}
\begin{split}
\|z(T)\|\leq C(R)\|w\|_{X_{0}}^{2}\leq C(R)\|v_{0}\|^{2}=C(R)\|u_{0}-\hat{u}_{0}\|^{2}\leq C(R)d\|u_{0}-\hat{u}_{0}\|
,
\end{split}
\end{equation*}
where we have used the second inequality in \eqref{18} and $\|u_{0}-\hat{u}_{0}\|\leq d.$ By taking $0<d\ll 1,$ we have
\begin{equation*}
\begin{split}
\|z(T)\|\leq \frac{q}{2}\|u_{0}-\hat{u}_{0}\|
.
\end{split}
\end{equation*}
Combining this with the first inequality of \eqref{18}, we prove \eqref{21}.
\par
If we define the resolving operator for the KdVB equation \eqref{26} as $\hat{u}=\hat{u}(h,\hat{u}_{0}),$ it is easy to see that
$$\Upsilon(h,\hat{u}_{0})v_0=\chi \Pi_{m}(\Psi(\hat{u}(h,\hat{u}_{0}))v_{0}),$$ the regularity and Lipschitz continuity of $\Upsilon$ can be proved by combining similar properties of the resolving operator for \eqref{26} and $\Psi$ as in Lemma \ref{pro5}.
\par
This completes the proof of Proposition \ref{pro6}.
\end{proof}

\subsection{Proof of Theorem \ref{MT}}
Multiplying the equation in \eqref{1} by $u$ and then performing integration by parts over $\mathbb{T},$ we get
\begin{equation*}
\begin{split}
\frac{1}{2}\frac{d}{dt}\|u(t)\|^{2}+\|u(t)\|_{1}^{2}=(h+\eta,u).
\end{split}
\end{equation*}
With the help of Poincar\'{e} inequality, interpolation inequality and Young inequality, for sufficiently small $\gamma,$ there exists a positive constant $\nu$ such that
\begin{equation}\label{6}
\begin{split}
\frac{d}{dt}\|u(t)\|^{2}+\nu\|u(t)\|^{2}\leq C\|h+\eta\|^{2}
.
\end{split}
\end{equation}
By applying the Gronwall's inequality to the second inequality in \eqref{6}, we have
\begin{equation*}
\begin{split}
\|u(t)\|^{2}\leq e^{-\nu t}\|u_{0}\|^{2}+\int_{0}^{t}e^{-\nu(t-s)}\|(h+\eta)(s)\|^{2}ds~{\rm{for~any}}~t\geq0,
\end{split}
\end{equation*}
this implies that there exists some constant $0<\kappa<1$ such that
\begin{equation*}
\begin{split}
\|u(T)\|\leq \kappa\|u_{0}\|+C_{1}\|h+\eta_{k}\|_{L^{2}(D_{T})}.
\end{split}
\end{equation*}
Let $r>0$ be so large that $\|h+\eta_{k}\|_{L^{2}(D_{T})}\leq r$ almost surly and $R\geq \frac{C_{1}r}{1-\kappa},$ then
for any $u_{0}\in B_{H}(R),$ we have $u(T)\in B_{H}(R).$
We define
\begin{equation}\label{25}
\begin{split}
X:=S(B_{H}(R),B_{L^{2}(D_{T})}(r)).
\end{split}
\end{equation}
Then, $X\subset H$ is a compact set such that
\begin{equation*}
\begin{split}
&\mathbb{P}\{S_{k}(u_{0},\eta)\in X~{\rm{for~any}}~k\geq0\}=1~{\rm{for~any}}~u_{0}\in X,\\
&\mathbb{P}\{S_{k}(u_{0},\eta)\in X~{\rm{for~any}}~k\geq~{\rm{some}}~k_{0}\}=1~{\rm{for~any}}~u_{0}\in H
.
\end{split}
\end{equation*}
This implies that the random flow generated by \eqref{1} possesses a compact invariant absorbing set $X,$
and $(u_{k},\mathbb{P}_{u})$ has at least one stationary measure $\mu,$ and any such measure is supported by $X.$
Therefore it suffices to prove the uniqueness of an invariant measure and the property of
exponential mixing for the restriction of $(u_{k},\mathbb{P}_{u})$ to $X.$ This will be done with the help of
Proposition \ref{proSC}.
\par
We take $E:=L^{2}(D_{T}),$ $X$ as in \eqref{25}, $\mathcal{K}=supp~\ell$ in Proposition \ref{proSC}. Condition (AC) implies that Hypotheses (ACP) in Proposition \ref{proSC} holds. We define $\Phi(u_{0},
\hat{u}_{0},\eta):=\Upsilon(\eta,\hat{u}_{0})(u_{0}-\hat{u}_{0})$, Proposition \ref{pro6} implies that Hypotheses (LS) in Proposition \ref{proSC} holds. Since the random variables $\eta_{k}$ satisfy Condition (DN), this implies that Hypotheses (D) in Proposition \ref{proSC} holds. By applying Proposition \ref{proSC} to \eqref{1}, we can prove Theorem \ref{MT}.

\subsection{An example for Condition (AC)}
\begin{proposition}
If $h=h(t,x)$ is a given function which is $T$-periodic in time, and $\|h\|_{H^{1}(D_{T})}$ is sufficiently small, there is a $\bar{u}\in H$ such that Condition (AC) holds.
\end{proposition}
\begin{proof}
By the similar analysis as in \cite{GP2018AA}, with the help of fixed point method, it is easy to show that there is a constant $\delta>0$ such that $\|h\|_{H^{1}(D_{T})}<\delta$, the system
$$u_{t}+Au +B(u)=h$$
has a unique solution $\tilde{u}$ defined throughout the real line and $T-$periodic in time. Moreover, it holds that
\begin{equation*}
\begin{split}
\sup\limits_{t\in \mathbb{R}}\|\tilde{u}(t)\|_{1}\leq C\delta\rightarrow0~~{\rm{}as}~~\delta\rightarrow0,
\end{split}
\end{equation*}
where $C>0$ is a constant independent of $\delta.$ This implies that $\tilde{u}$ is globally exponentially stable as $t\rightarrow+\infty.$ Therefore, by taking $\bar{u}=\tilde{u}(0)$ for any positive
constants $R$ and $\varepsilon$ one can find an integer $l\geq1$ such that \eqref{24} holds with $\zeta_{1}=\zeta_{2}=\cdots=\zeta_{l}=0$. Since $\mathcal{K}$ contains the zero element, we see that Condition (AC) is satisfied.
\end{proof}

\section{Foia\c{s}-Prodi estimate for KdVB equation}
\subsection{Foia\c{s}-Prodi estimate for KdVB equation}
Let us consider the following two systems
\begin{eqnarray}
&&\label{My}du+[Au+B(u)]dt=hdt+g(u)dW,\\
&&\label{Mz}dv+[Av+B(v)-\lambda P_{N}(u-v)]dt=hdt+g(v)dW.
\end{eqnarray}
\par~~
\par
Now, we are in a position to present the Foia\c{s}-Prodi estimate for KdVB equation.
\begin{theorem}(Foia\c{s}-Prodi estimate for KdVB equation)\label{FPE}
Let $h=h(x)$, Conditions (g1)-(g2) hold and $L_{3}<1$. Let $u$ and $v$ be the
solutions to \eqref{My} and \eqref{Mz} with $u_{0},v_{0}\in H$, then there exists a positive integer $N$ large enough and $\lambda=\frac{\lambda_{N}}{2}$
such that for any
\begin{equation}\label{54}
\begin{split}
p\in
\left\{
\begin{array}{lll}
(2,+\infty)\\
(2,1+\frac{1}{L_{3}^{2}})
\end{array}
\begin{array}{lll}
under~(g2)(i)~or~(g2)(ii),\\
under~(g2)(iii),
\end{array}
\right.
\end{split}
\end{equation}
we have
\begin{equation}\label{62}
\begin{split}
\mathbb{E}\|u(t)-v(t)\|^{2}\leq
\left\{
\begin{array}{lll}
e^{C(1+\|u_{0}\|^{2p}+\|v_{0}\|^{2p})}e^{-Ct} \\
e^{C(1+\|u_{0}\|^{2p}+\|v_{0}\|^{2p})}\frac{1}{t^{\frac{p}{4}-\frac{1}{2}}}
\end{array}
\right.
\begin{array}{lll}
under~(g2)(i),\\
under~(g2)(ii)~or~(g2)(iii).
\end{array}
\end{split}
\end{equation}
\end{theorem}
Foia\c{s}-Prodi estimate was firstly established in \cite{FPE}, now, it is a powerful tool to
establish the ergodicity for SPDEs, and it is often used to compensate the degeneracy of the noise, see \cite{KS12, M2014} and the references therein.
The Foia\c{s}-Prodi estimate of KdV equation on bounded domain was firstly established in \cite{Glat2021}.
\par
Now, we start to prove Theorem \ref{FPE}.
\subsection{Moment estimates and estimates in probability}
\begin{lemma}\label{L2}
Let $h=h(x)$, (g1)-(g2) hold and $L_{3}<1$.
It holds that
\begin{equation}\label{45}
\begin{split}
\mathbb{E}\|u(t)\|^{2}+\frac{3}{2}\mathbb{E}\int_{0}^{t}\|u(s)\|_{1}^{2}ds\leq \|u_{0}\|^{2}+bt,
\end{split}
\end{equation}
where
\begin{equation*}
\begin{split}
b=
\left\{
\begin{array}{lll}
K_{1}^{2}+C\|h\|_{-1}^{2}\\
2K_{2}^{2}+C(\varrho)+C\|h\|_{-1}^{2}\\
CK_{3}^{2}+C\|h\|_{-1}^{2}
\end{array}
\right.
\begin{array}{lll}
under~(g2)(i),\\
under~(g2)(ii),\\
under~(g2)(iii).
\end{array}
\end{split}
\end{equation*}
\end{lemma}

\begin{proof}By applying It\^{o} formula to $\|u(t)\|^{2}$, we have
\begin{equation*}
\begin{split}
\frac{1}{2}d\|u(t)\|^{2}+\|u(t)\|_{1}^{2}dt=[\frac{1}{2}\|g(u)\|_{HS}^{2}+(u,h)]dt+(u,g(u)dW).
\end{split}
\end{equation*}
It follows from Condition (g2) and Young inequality that for any $\varepsilon>0$
\begin{equation*}
\begin{split}
\|g(u)\|_{HS}^{2}\leq
\left\{
\begin{array}{lll}
K_{1}^{2}\\
2K_{2}^{2}+C(\varepsilon,\varrho)+\varepsilon L_{2}^{2}\|u\|_{1}^{2}\\
(1+\frac{1}{\varepsilon})K_{3}^{2}+(1+\varepsilon)L_{3}^{2}\|u\|_{1}^{2}
\end{array}
\right.
\begin{array}{lll}
under~(g2)(i),\\
under~(g2)(ii),\\
under~(g2)(iii).
\end{array}
\end{split}
\end{equation*}
Since $2|(u,h)|\leq \frac{1}{8}\|u\|_{1}^{2}+C\|h\|_{-1}^{2},$ this implies that
\begin{equation}\label{44}
\begin{split}
\|u(t)\|^{2}+\frac{3}{2}\int_{0}^{t}\|u(s)\|_{1}^{2}ds\leq \|u_{0}\|^{2}+bt+M(t),
\end{split}
\end{equation}
where $M(t)=2\int_{0}^{t}(u,g(u)dW).$ Therefore, by taking the
expected values on both sides of \eqref{44}, we get \eqref{45}.
\par
This completes the proof.
\end{proof}

\begin{lemma}\label{L3}
Let $h=h(x)$, (g1)-(g2) hold and $L_{3}<1$. For any
\begin{equation}\label{55}
\begin{split}
p
\left\{
\begin{array}{lll}
\in [2,+\infty)\\
\in [2,1+\frac{1}{L_{3}^{2}})
\end{array}
\right.
\begin{array}{lll}
under~(g2)(i)~or~(g2)(ii),\\
under~(g2)(iii),
\end{array}
\end{split}
\end{equation}
there exist some constants $\gamma_{p}>0,C_{p}>0$ such that
\begin{equation}\label{46}
\begin{split}
\mathbb{E}\|u(t)\|^{p}\leq \|u_{0}\|^{p}e^{-\gamma_{p}t}+C_{p}, ~\forall t\geq0.
\end{split}
\end{equation}
\end{lemma}

\begin{proof}
By applying It\^{o} formula to $\|u(t)\|^{p}$, we have
\begin{equation*}
\begin{split}
d\|u(t)\|^{p}+&p\|u(t)\|^{p-2}\|u(t)\|_{1}^{2}dt=\\
&\frac{p(p-1)}{2}\|u(t)\|^{p-2}\|g(u)\|_{HS}^{2}dt+p\|u(t)\|^{p-2}(u,h)dt+p\|u(t)\|^{p-2}(u,g(u)dW).
\end{split}
\end{equation*}
Since $2|(u,h)|\leq \frac{1}{8}\|u\|_{1}^{2}+C\|h\|_{-1}^{2},$ this implies that for any $\varepsilon,\eta>0$ there exists a constant $C_{\varepsilon,\eta}$ such that
\begin{equation*}
\begin{split}
2p\|u(t)\|^{p-2}(u,h)\leq p\eta\|u(t)\|^{p-2}\|u(t)\|_{1}^{2}+\frac{\varepsilon}{2}\|u(t)\|^{p}+C_{\varepsilon,\eta}\|h\|_{-1}^{p}.
\end{split}
\end{equation*}
It follows from Condition (g2) that for any $\varepsilon>0$
\begin{equation*}
\begin{split}
\frac{p(p-1)}{2}\|u(t)\|^{p-2}\|g(u)\|_{HS}^{2}\leq
\left\{
\begin{array}{lll}
C+\frac{\varepsilon}{2}\|u(t)\|^{p}\\
C+\frac{p(p-1)}{2}(1+\varepsilon)L_{3}^{2}\|u(t)\|^{p}
\end{array}
\begin{array}{lll}
under~(g2)(i)~or~(g2)(ii),\\
under~(g2)(iii).
\end{array}
\right.
\end{split}
\end{equation*}
\par
If Condition (g2)(i) or (g2)(ii) hold, by taking $\eta$ small enough, the above estimates imply that
\begin{equation*}
\begin{split}
d\|u(t)\|^{p}+\frac{p}{2}\|u(t)\|^{p-2}\|u(t)\|_{1}^{2}dt\leq
C+\varepsilon\|u(t)\|^{p}+C_{\varepsilon}\|h\|_{-1}^{p}+p\|u(t)\|^{p-2}(u,g(u)dW).
\end{split}
\end{equation*}
If $\varepsilon$ is sufficiently small and by Poincar\'{e} inequality,
we have
\begin{equation}\label{47}
\begin{split}
d\|u(t)\|^{p}+\gamma_{p}\|u(t)\|^{p}dt\leq
C+C\|h\|_{-1}^{p}+p\|u(t)\|^{p-2}(u,g(u)dW).
\end{split}
\end{equation}
Therefore, by taking the expected values on both sides of \eqref{47}, with the help of Gronwall inequality, we get \eqref{46}.
\par
If Condition (g2)(iii) holds, by taking $\eta$ small enough, the above estimates imply that
\begin{equation*}
\begin{split}
&d\|u(t)\|^{p}+\frac{p}{2}\|u(t)\|^{p-2}\|u(t)\|_{1}^{2}dt\\
\leq &C+\frac{\varepsilon}{2}\|u(t)\|^{p}+\frac{p(p-1)}{2}(1+\varepsilon)L_{3}^{2}\|u(t)\|^{p}
+C_{\varepsilon}\|h\|_{-1}^{p}+p\|u(t)\|^{p-2}(u,g(u)dW).
\end{split}
\end{equation*}
It follows from Poincar\'{e} inequality that
\begin{equation}\label{48}
\begin{split}
d\|u(t)\|^{p}+[\frac{p}{2}-\frac{\varepsilon}{2}-\frac{p(p-1)}{2}(1+\varepsilon)L_{3}^{2}]\|u(t)\|^{p}dt\leq
C+C\|h\|_{-1}^{p}+p\|u(t)\|^{p-2}(u,g(u)dW).
\end{split}
\end{equation}
If $\varepsilon$ is sufficient small, we have $\frac{p}{2}-\frac{\varepsilon}{2}-\frac{p(p-1)}{2}(1+\varepsilon)L_{3}^{2}>0,$
by taking the
expected values on both sides of \eqref{48}, with the help of Gronwall inequality, we get \eqref{46}.
\par
This completes the proof.
\end{proof}

\begin{lemma}\label{L4}
Let $h=h(x)$, Conditions (g1)-(g2) hold and $L_{3}<1$. For any $p$ in \eqref{55}, there exists some constant $C_{p}>0$ such that
\begin{equation*}
\begin{split}
\sup\limits_{t\geq0}\mathbb{E}\|v(t)\|^{p}\leq C_{p}(\|u_{0}\|^{p}+\|v_{0}\|^{p}+1).
\end{split}
\end{equation*}
\end{lemma}
\begin{proof}
By applying It\^{o} formula to $\|v(t)\|^{p}$, we have
\begin{equation}\label{51}
\begin{split}
&d\|v(t)\|^{p}+p\|v(t)\|^{p-2}\|v(t)\|_{1}^{2}dt=\\
&[\frac{p(p-1)}{2}\|v(t)\|^{p-2}\|g(v)\|_{HS}^{2}+p\|v(t)\|^{p-2}(v,h)+p\lambda\|v(t)\|^{p-2}(v,P_{N}(u-v))]dt\\
&+p\|v(t)\|^{p-2}(v,g(v)dW).
\end{split}
\end{equation}
We only need to note the fact
\begin{equation*}
\begin{split}
\|v\|^{p-2}(v,P_{N}(u-v))\leq\|v\|^{p-2}(\frac{1}{2\varepsilon}\|u\|^{2}+\frac{\varepsilon}{2}\|v\|^{2})
\end{split}
\end{equation*}
for any $\varepsilon>0.$ By choosing $\varepsilon$ sufficiently small,
there exist $\bar{a}_{p}>0$ and $C_{p}>0$ such that
\begin{equation*}
\begin{split}
d\|v(t)\|^{p}+\bar{a}_{p}\|v(t)\|^{p}dt\leq C_{p}[(1+\|u(t)\|^{p})dt+\|v(t)\|^{p-2}(v,g(v)dW)].
\end{split}
\end{equation*}
By the same argument as in the proofs of Lemma \ref{L2} and Lemma \ref{L3}, we can finish the proof.
\end{proof}
\par~~
\par
Now, we establish some estimates in probability.
\begin{proposition}\label{pro13}
Let $h=h(x)$, Conditions (g1)-(g2) hold and $L_{3}<1$.
\par
(1) If (g2)(i) holds,
\begin{equation}
\begin{split}
\mathbb{P}\left(\sup\limits_{t\geq 0}[\|u(t)\|^{2}+\int_{0}^{t}\|u(s)\|_{1}^{2}ds-\|u_{0}\|^{2}-bt]\geq R\right)\leq e^{-\alpha R},
\end{split}
\end{equation}
where $\alpha=\frac{1}{8K_{1}^{2}}$.
\par
(2) If (g2)(ii) or (g2)(iii) holds, for any $p$ in \eqref{54}, we have
\begin{equation}
\begin{split}
\mathbb{P}\left(\sup\limits_{t\geq T}[\|u(t)\|^{2}+\int_{0}^{t}\|u(s)\|_{1}^{2}ds-\|u_{0}\|^{2}-C_{b}(t+1)]\geq R\right)\leq \frac{C(1+\|u_{0}\|^{2p})}{(T+R)^{\frac{p}{2}-1}},
\end{split}
\end{equation}
for all $T\geq0,R>0,$ where $C_{b}=\max\{b+1,2\}.$
\end{proposition}
\begin{proof}
Let us first recall \eqref{44}
\begin{equation*}
\begin{split}
\|u(t)\|^{2}+\frac{3}{2}\int_{0}^{t}\|u(s)\|_{1}^{2}ds\leq \|u_{0}\|^{2}+bt+M(t).
\end{split}
\end{equation*}
\par
(1) Under Condition (g2)(i), we can see that the quadratic variation of $M$ has the following estimate
\begin{equation*}
\begin{split}
[M](t)\leq 4K_{1}^{2}\int_{0}^{t}\|u(s)\|^{2}ds\leq 4K_{1}^{2}\int_{0}^{t}\|u(s)\|_{1}^{2}ds,
\end{split}
\end{equation*}
this implies that
\begin{equation*}
\begin{split}
\|u(t)\|^{2}+\int_{0}^{t}\|u(s)\|_{1}^{2}ds\leq \|u_{0}\|^{2}+bt+(M(t)-\alpha[M](t)).
\end{split}
\end{equation*}
By applying the exponential martingale inequality, we have
\begin{equation*}
\begin{split}
\mathbb{P}\left(\sup\limits_{t\geq 0}[\|u(t)\|^{2}+\int_{0}^{t}\|u(s)\|_{1}^{2}ds-\|u_{0}\|^{2}-bt]\geq R\right)
&\leq\mathbb{P}(\sup\limits_{t\geq 0}[M(t)-\alpha[M](t)]\geq R)\\
&\leq e^{-\alpha R}.
\end{split}
\end{equation*}
\par
(2) Under Condition (g2)(ii) or (g2)(iii), we have
\begin{equation*}
\begin{split}
\mathbb{P}(\sup\limits_{t\geq T}[\|u(t)\|^{2}+\int_{0}^{t}\|u(s)\|_{1}^{2}ds-\|u_{0}\|^{2}-C_{b}(t+1)]\geq R)
\leq \mathbb{P}(\sup\limits_{t\geq T}[M(t)-t-2]\geq R).
\end{split}
\end{equation*}
We define $M^{*}(t)=\sup\limits_{s\in[0,t]}|M(s)|.$ By means of the Young inequality, under Condition (g2)(ii) or (g2)(iii), we can
estimate the quadratic variation $[M](t)$ as follows
\begin{equation*}
\begin{split}
[M](t)\leq 4\int_{0}^{t}\|u(s)\|^{2}\|g(u(s))\|_{HS}^{2}ds\leq C\int_{0}^{t}(1+\|u(s)\|^{4})ds.
\end{split}
\end{equation*}
According to the Burkholder-Davis-Gundy and H\"{o}lder inequalities and Lemma \ref{L3}, we have
\begin{equation}\label{49}
\begin{split}
\mathbb{E}[M^{*}(t)]^{p}
&\leq C_{p}\mathbb{E}([M](t))^{\frac{p}{2}}\leq C\mathbb{E}(\int_{0}^{t}(1+\|u(s)\|^{4})ds)^{\frac{p}{2}}\\
&\leq C_{p}t^{\frac{p-2}{2}}\mathbb{E}\int_{0}^{t}(1+\|u(s)\|^{2p})ds\leq C_{p}(t+1)^{\frac{p}{2}}(1+\|u_{0}\|^{2p}).
\end{split}
\end{equation}
\par
Noting the following fact that for all $T\geq0,R>0,$
\begin{equation*}
\begin{split}
\{\sup\limits_{t\geq T}[M(t)-t-2]\geq R\}
\subset &\bigcup\limits_{m\geq [T]}\{\sup\limits_{t\in [m,m+1)}[M(t)-t-2]\geq R\}\\
\subset &\bigcup\limits_{m\geq [T]}\{M^{*}(m+1)\geq R+m+2\},
\end{split}
\end{equation*}
it follows from the above analysis, the Chebyshev inequality, \eqref{49} and noting $p>2$ that
\begin{equation*}
\begin{split}
\mathbb{P}\{\sup\limits_{t\geq T}[M(t)-t-2]\geq R\}
&\leq\mathbb{P}(\bigcup\limits_{m\geq [T]}\{M^{*}(m+1)\geq R+m+2\})\\
&\leq\sum\limits_{m\geq [T]}\mathbb{P}(M^{*}(m+1)\geq R+m+2)\\
&\leq\sum\limits_{m\geq [T]}\frac{\mathbb{E}(M^{*}(m+1))^{p}}{(R+m+2)^{p}}\\
&\leq(1+\|u_{0}\|^{2p})\sum\limits_{m\geq [T]}\frac{(m+2)^{\frac{p}{2}}}{(R+m+2)^{p}}\\
&\leq(1+\|u_{0}\|^{2p})\sum\limits_{m\geq [T]}\frac{1}{(R+m+2)^{\frac{p}{2}}}\\
&\leq C\frac{1+\|u_{0}\|^{2p}}{(R+T)^{\frac{p}{2}-1}}.
\end{split}
\end{equation*}
\par
This completes the proof.

\end{proof}

\subsection{Proof of Theorem \ref{FPE}}

\begin{lemma}\label{L1}
Let $h=h(x)$, $\lambda\geq\frac{\lambda_{N}}{2},$ Conditions (g1)-(g2) hold and $L_{3}<1$.
Then we have
\begin{equation*}
\begin{split}
\mathbb{E}e^{\Gamma(t\wedge\tau)}\|w(t\wedge\tau)\|^{2}+\frac{\lambda_{N}}{2}\mathbb{E}\int_{0}^{t\wedge\tau}e^{\Gamma(s)}\|w(s)\|^{2}ds\leq \|w(0)\|^{2},
\end{split}
\end{equation*}
where $\Gamma(t)$ is defined in \eqref{50}.
\end{lemma}
\begin{proof}
Let $w:=u-v,$ then $w$ satisfies that
\begin{equation*}
\begin{split}
dw+[Aw+(B(u)-B(v))+\lambda P_{N}w]dt=(g(u)-g(v))dW.
\end{split}
\end{equation*}
By applying It\^{o} formula to $\|w(t)\|^{2}$, we have
\begin{equation*}
\begin{split}
\frac{1}{2}d\|w\|^{2}+[\|w\|_{1}^{2}+\lambda\|P_{N}w\|^{2}]dt=&[-(B(u)-B(v),w)+\frac{1}{2}\|g(u)-g(v)\|_{HS}^{2}]dt\\
&+(w,(g(u)-g(v))dW).
\end{split}
\end{equation*}
Noting the fact $B(u)-B(v)=(uw)_{x}-\frac{1}{2}(w^{2})_{x},$ we have
\begin{equation*}
\begin{split}
-(B(u)-B(v),w)=-((uw)_{x},w)=(uw,w_{x}),
\end{split}
\end{equation*}
there exists a constant $C_{0}$ such that
\begin{equation*}
\begin{split}
|(B(u)-B(v),w)|\leq \frac{C_{0}}{2}\|u\|_{1}^{2}\|w\|^{2}+\frac{1}{2}\|w\|_{1}^{2},
\end{split}
\end{equation*}
this leads to
\begin{equation*}
\begin{split}
\frac{1}{2}d\|w\|^{2}+[\frac{1}{2}\|w\|_{1}^{2}+\lambda\|P_{N}w\|^{2}]dt
\leq(\frac{L_{g}^{2}}{2}+\frac{C_{0}}{2}\|u\|_{1}^{2})\|w\|^{2}dt+dM(t),
\end{split}
\end{equation*}
where $M(t)=\int_{0}^{t}(w(s),(g(u)-g(v))dW).$
If we choose $\lambda=\frac{\lambda_{N}}{2}$ and fixed, according to the generalized inverse Poincar\'{e} inequality $\lambda_{N}\|Q_{N}w(t)\|^{2}\leq\|w(t)\|_{1}^{2}$, we have
$$\frac{1}{2}\|w(t)\|_{1}^{2}+\lambda\|P_{N}w(t)\|^{2}\geq \frac{\lambda_{N}}{2}\|w(t)\|^{2},$$
this leads to
\begin{equation*}
\begin{split}
d\|w(t)\|^{2}+[\lambda_{N}-L_{g}^{2}-C_{0}\|u(t)\|_{1}^{2}]\|w(t)\|^{2}dt\leq 2dM(t).
\end{split}
\end{equation*}
We define
\begin{equation}\label{50}
\begin{split}
\Gamma(t):=(\frac{\lambda_{N}}{2}-L_{g}^{2})t-C_{0}\int_{0}^{t}\|u(s)\|_{1}^{2}ds,
\end{split}
\end{equation}
thus, it holds that
\begin{equation*}
\begin{split}
d\|w(t)\|^{2}+[\frac{\lambda_{N}}{2}+\Gamma^{'}(t)]\|w(t)\|^{2}dt\leq 2dM(t),
\end{split}
\end{equation*}
this leads to
\begin{equation*}
\begin{split}
d(e^{\Gamma(t)}\|w(t)\|^{2})+\frac{\lambda_{N}}{2}e^{\Gamma(t)}\|w(t)\|^{2}dt\leq 2e^{\Gamma(t)}dM(t),
\end{split}
\end{equation*}
integrating the above estimate from $0$ to $t\wedge\tau$ and taking expectation, this completes the proof.
\end{proof}
\par
Now, we are in a position to prove Theorem \ref{FPE}.
\begin{proof}[Proof of Theorem \ref{FPE}]The proof of Theorem \ref{FPE} is divided into several steps.
\par
\textbf{Step 1.}
Let us introduce the following stopping time
\begin{equation*}
\begin{split}
\tau_{R,\beta}:=
\left\{
\begin{array}{lll}
\inf \{t\geq0:C_{0}\int_{0}^{t}\|u(s)\|_{1}^{2}ds-(\frac{\lambda_{N}}{4}-L_{g}^{2})t-\beta\geq R\}\\
+\infty,~~~~{\rm{if}}~C_{0}\int_{0}^{t}\|u(s)\|_{1}^{2}ds-(\frac{\lambda_{N}}{4}-L_{g}^{2})t-\beta< R,~\forall t\geq0.
\end{array}
\right.
\end{split}
\end{equation*}
We can see that if $\tau_{R,\beta}=+\infty,$ then $\frac{\lambda _{N}}{4}t- (R+\beta)\leq \Gamma(t)$ for $\forall t\geq0$.
With the help of Lemma \ref{L1}, this leads to the following fact: For any $R,\beta>0,$ we have
\begin{equation}\label{2}
\begin{split}
\mathbb{E}(\textbf{1}_{(\tau_{R,\beta}=+\infty)}\|w(t)\|^{2})\leq e^{R+\beta-\frac{\lambda_{N}}{4}t}\|w(0)\|^{2}.
\end{split}
\end{equation}
\par
\textbf{Step 2.} We want to estimate $\mathbb{P}(\tau_{R,\beta}<+\infty)$ in terms of $R.$
\par
Indeed, we define the set
\begin{equation*}
\begin{split}
A_{R,\beta}:=\left\{\sup_{t\geq0}[C_{0}\int_{0}^{t}\|u(s)\|_{1}^{2}ds-(\frac{\lambda_{N}}{4}-L_{g}^{2})t]-\beta\geq R\right\}.
\end{split}
\end{equation*}
Noting the fact $\mathbb{P}(\tau_{R,\beta}<+\infty)\leq\mathbb{P}(A_{R,\beta})$,
by taking $N,\beta$ large enough, with the help of Proposition \ref{pro13},
we can prove that there exists a positive integer $N,\beta$ large enough
such that
\begin{equation}\label{7}
\begin{split}
\mathbb{P}(\tau_{R,\beta}<+\infty)\leq
\left\{
\begin{array}{lll}
2e^{-CR} \\
\frac{C(1+\|u_{0}\|^{2p}+\|v_{0}\|^{2p})}{R^{\frac{p}{2}-1}}
\end{array}
\right.
\begin{array}{lll}
under~(g2)(i),\\
under~(g2)(ii)~or~(g2)(iii),
\end{array}
\end{split}
\end{equation}
for any $p$ in \eqref{54}, where $C$ is a positive constant independent of $R, \beta$ and $u_{0},v_{0}.$
\par
\textbf{Step 3.} Proof of \eqref{62}.
\par
Indeed, it follows from \eqref{2} and the following fact that
\begin{equation*}
\begin{split}
\mathbb{E}\|u(t)-v(t)\|^{2}
&=\mathbb{E}(\textbf{1}_{(\tau_{R,\beta}=+\infty)}\|u(t)-v(t)\|^{2})+\mathbb{E}(\textbf{1}_{(\tau_{R,\beta}<+\infty)}\|u(t)-v(t)\|^{2})\\
&\leq e^{R+\beta-\frac{\lambda_{N}}{4}t}\|u_{0}-v_{0}\|^{2}+(\mathbb{P}(\tau_{R,\beta}<+\infty))^{\frac{1}{2}}(\mathbb{E}\|u(t)-v(t)\|^{4})^{\frac{1}{2}}\\
&\leq C(1+\|u_{0}\|^{2}+\|v_{0}\|^{2})(e^{R+\beta-\frac{\lambda_{N}}{8}t}+(\mathbb{P}(\tau_{R,\beta}<+\infty))^{\frac{1}{2}}).
\end{split}
\end{equation*}
We take $N,\beta$ large enough and $R=\frac{\lambda_{N}}{16}t,$ according to
\eqref{7}, we can obtain \eqref{62}.
\par
This completes the proof.
\end{proof}

\section{Proofs of Theorem \ref{MT2} and Theorem \ref{MT3}}
In this section, we apply the
asymptotic coupling method to prove Theorem \ref{MT2} and Theorem \ref{MT3}.
We now state the abstract results from \cite{AC1} and \cite{AC2} in the form that best fits our context.
\subsection{Asymptotic coupling method}
Let $H$ be a Polish space with a metric $\rho$, $\mathcal{B}(H)$ denote the Borel $\sigma-$algebra on $H,$ and
$P$ be a Markov transition kernel on $H$. Namely, we suppose that $P : H\times \mathcal{B}(H)\rightarrow[0, 1]$ such that $P(u,\cdot)$
is a probability measure for any given $u\in H$ and that $P(\cdot,A)$ is a measurable function for any fixed
$A\in\mathcal{B}(H)$. $P$ acts on bounded measurable observables $\varphi: H\rightarrow \mathbb{R}$ and
Borel probability measures $\mu$ as
\begin{equation*}
\begin{split}
P\varphi(u)=\int\varphi(v)P(u,dv),~~\mu P(A)=\int P(u,A)\mu(du),
\end{split}
\end{equation*}
respectively. A probability measure $\mu$ on $\mathcal{B}(H)$ is invariant if $\mu P=\mu$.

\par
We introduce the space of one-sided infinite sequences
\begin{equation*}
\begin{split}
H^{\mathbb{N}}=\{u:\mathbb{N}\rightarrow H\}=\{u=(u_{1},u_{2},\cdots): u_{i}\in H\},
\end{split}
\end{equation*}
with its Borel $\sigma-$field $\mathcal{B}(H^{\mathbb{N}})$, we denote by $\mathcal{P}(H^{\mathbb{N}})$ the collections of Borel probability measures on $H^{\mathbb{N}}$.
For given $\mu,\nu\in \mathcal{P}(H^{\mathbb{N}})$, we define
\begin{equation*}
\begin{split}
\mathcal{C}(\mu,\nu):=\{\xi\in \mathcal{P}(H^{\mathbb{N}}\times H^{\mathbb{N}}):\pi_{1}(\xi)=\mu,\pi_{2}(\xi)=\nu\},
\end{split}
\end{equation*}
where $\pi_{i}(\xi)$ denotes the $i-$th marginal distribution of $\xi, i=1,2.$ Any $\xi\in \mathcal{C}(\mu,\nu)$ is
called a coupling for $\mu,\nu.$ Recall that $\mu\ll\nu$ means that $\mu$ is absolutely continuous w.r.t. $\nu$, and
$\mu\sim\nu$ means that $\mu$ and $\nu$ are equivalent, i.e., mutually absolutely continuous. We
define
\begin{equation*}
\begin{split}
&\tilde{\mathcal{C}}(\mu,\nu):=\{\xi\in \mathcal{P}(H^{\mathbb{N}}\times H^{\mathbb{N}}):\pi_{1}(\xi)\sim\mu,\pi_{2}(\xi)\sim\nu\},\\
&\hat{\mathcal{C}}(\mu,\nu):=\{\xi\in \mathcal{P}(H^{\mathbb{N}}\times H^{\mathbb{N}}):\pi_{1}(\xi)\ll\mu,\pi_{2}(\xi)\ll\nu\},\\
\end{split}
\end{equation*}
and call any probability measure from the classes $\tilde{\mathcal{C}}(\mu,\nu),\hat{\mathcal{C}}(\mu,\nu)$ a generalized
coupling for $\mu,\nu.$ We define
\begin{equation*}
\begin{split}
&D:=\{(x,y)\in H^{\mathbb{N}}\times H^{\mathbb{N}} : \lim\limits_{n\rightarrow\infty}\|x(n)-u(n)\|=0\},\\
&D_{\varepsilon}^{n}:=\{(x,y)\in H^{\mathbb{N}}\times H^{\mathbb{N}} : \|x(n)-u(n)\|\leq \varepsilon\},\\
\end{split}
\end{equation*}
where $\varepsilon>0, n\in \mathbb{N}.$ The set of test functions
\begin{equation*}
\begin{split}
\mathcal{G}:=\{\varphi\in C_{b}(H):\sup\limits_{x\neq y}\frac{|\varphi(x)-\varphi(y)|}{\|x-y\|}<\infty\}
\end{split}
\end{equation*}
which is a \textit{determining measure set} in $H,$ namely, if $\mu,\nu\in \mathcal{P}(H)$ are such that $\int_{H}\varphi(u)\mu(du)=\int_{H}\varphi(u)\nu(du)$
for all $\varphi\in \mathcal{G},$ then we have $\mu=\nu.$
\begin{theorem}\label{AC1}(See \cite{AC1})
If $\mathcal{G}$ determines measures on $H$ and that $D\subset H^{\mathbb{N}}\times H^{\mathbb{N}}$ is measurable. If for each $u_{0},v_{0}\in H$, there exists a generalized coupling $\xi_{u_{0},v_{0}}\in \mathcal{\hat{C}}(\mathbb{P}_{u_{0}},\mathbb{P}_{z_{0}})$ such that $\xi_{u_{0},v_{0}}(D)> 0,$ then there is at most one $P-$invariant
probability measure $\mu\in \mathcal{P}(H).$

\end{theorem}

\begin{theorem}\label{AC2}(See \cite{AC2})
If the transition semigroup $\{\mathcal{P}_{t}\}_{t\geq 0}$ associated with \eqref{} is
a Feller semigroup on $H,$ and for any $u_{0},v_{0}\in H$ there exists some $\xi_{u_{0},v_{0}}\in \mathcal{\hat{C}}(\mathbb{P}_{u_{0}},\mathbb{P}_{v_{0}})$
such that $\pi_{1}(\xi_{u_{0},v_{0}})\sim\mathbb{P}_{u_{0}}$, and for any $\varepsilon>0$
$
\lim\limits_{n\rightarrow\infty}\xi_{u_{0},v_{0}}(D_{\varepsilon}^{n})=1.
$
Then, there exists at most one invariant probability measure, and, if such a measure
$\mu$ exists, then for any $u_{0}\in H,$
\begin{equation*}
\begin{split}
\lim\limits_{n\rightarrow\infty}\|\mathcal{P}^{*}_{t}\delta_{u_{0}}-\mu\|_{L}^{*}=0.
\end{split}
\end{equation*}

\end{theorem}

\subsection{Proof of Theorem \ref{MT2}}
We first introduce the nudged stopped equation.
\par
Let Condition (g3) hold for some $M\geq N_{0}.$ Let $y$ and $z$ be the solutions to \eqref{My} and \eqref{Mz} with initial dates $u_{0},v_{0}\in H,$
respectively. We define the shift by
\begin{equation}\label{56}
\begin{split}
h(t):=\lambda f(u(t))P_{N}(u(t)-v(t)),~~t\geq 0
\end{split}
\end{equation}
and a stopping time
\begin{equation*}
\begin{split}
\sigma_{K}:=\inf\{t\geq 0:\int_{0}^{t}\|P_{N}(u(s)-v(s))\|^{2}ds\geq K\},
\end{split}
\end{equation*}
where $K>0$ will be chosen in a suitable way later on. We introduce the following nudged stopped equation
\begin{eqnarray}\label{Mz1}
\begin{array}{l}
\left\{
\begin{array}{llll}
d\tilde{v}+[A\tilde{v}+B(\tilde{v})]dt=hdt+g(\tilde{v})d\tilde{W}
\\ \tilde{v}(x,0)=v_{0}(x)
\end{array}
\right.
\end{array}
\begin{array}
{lll}{\rm{in}}~D,
\\{\rm{in}}~\mathbb{T},
\end{array}
\end{eqnarray}
where $\tilde{W}:=W(t)+\int_{0}^{t}h(s)\mathbf{1}_{s\leq \sigma_{K}}ds.$ We denote by $\Psi_{u_{0}}$ and $\tilde{\Psi}_{u_{0},v_{0}}$ the measurable maps induced by solutions to
\eqref{My} and \eqref{Mz1}, respectively, that map an underlying probability space $(\Omega,\mathcal{F},\mathbb{P})$
to $C([0,+\infty);H).$ The laws of solutions of \eqref{My} and \eqref{Mz1} are given by $\mathbb{P}\Psi_{u_{0}}^{-1}$ and $\mathbb{P}(\tilde{\Psi}_{u_{0},v_{0}})^{-1}$ respectively.

\begin{proposition}\label{pro16}
Let $h=h(x)$ and Conditions (g1)-(g2) hold, $N$ be appearing in \eqref{56}.
Then for any $K>0,\lambda>0,$ the laws of solutions to \eqref{My} and \eqref{Mz1} are mutually absolutely continuous,
namely, $\mathbb{P}\Psi_{z_{0}}^{-1}\sim\mathbb{P}(\tilde{\Psi}_{u_{0},v_{0}})^{-1}$ as measures on $C([0,+\infty);H)$.
\end{proposition}
\begin{proof}
It follows from the definition of $\sigma_{K}$ and Condition (g3) that
\begin{equation*}
\begin{split}
\int_{0}^{\infty}\|h(s)\|_{U}^{2}\mathbf{1}_{s\leq \sigma_{K}}ds
&\leq\lambda^{2}(\sup\limits_{u\in H}\|f(u)\|_{L(H, U)}^{2})\int_{0}^{\infty}\|P_{N}(u(s)-v(s))\|^{2}\mathbf{1}_{s\leq \sigma_{K}}ds\\
&\leq\lambda^{2}(\sup\limits_{u\in H}\|f(u)\|_{L(H, U)}^{2})K,
\end{split}
\end{equation*}
the drift $h(s) \mathbf{1}_{s\leq \sigma_{K}}$ satisfies the Novikov condition
\begin{equation*}
\begin{split}
\mathbb{E}[\exp(\frac{1}{2}\int_{0}^{\infty}\|h(s)\|_{U}^{2}\mathbf{1}_{s\leq \sigma_{K}}ds)]<\infty.
\end{split}
\end{equation*}
By the Girsanov Theorem we infer that there exists a probability measure $Q$
on $C([0,+\infty);H)$  such that under $Q$, $\tilde{W}$ is a $U-$valued Wiener process on the time
interval $[0,+\infty)$. It follows that the law of the solution to the nudged stopped equation
\eqref{Mz1} is equivalent on $C([0,+\infty);H)$ to the law of the solution to the equation \eqref{My}
with initial condition $v_{0}$.
\end{proof}
\begin{remark}
According to the uniqueness of the solution
of equation \eqref{Mz}, on the set $\{\sigma_{K}=\infty\}$ we have that $v=\tilde{v}$, $\mathbb{P}-$a.s., where $v$ is the
solution of the nudged equation \eqref{Mz}.
\end{remark}

\begin{proposition}\label{pro15}
Let $h=h(x)$ and Conditions (g1)-(g2) hold.
Then there exists an integer $N_{0}\geq1$ and $\lambda=\lambda(N_{0})$, such that for $N\geq N_{0},$
it holds that
\begin{equation*}
\begin{split}
\mathbb{P}(\lim\limits_{n\rightarrow\infty}\|u(n)-\tilde{v}(n)\|=0)>0.
\end{split}
\end{equation*}
\end{proposition}
\begin{proof}
For $R,m>0$ to be chosen later on, we define
\begin{equation*}
\begin{split}
E_{R,m}:=\{\int_{0}^{m}\|P_{N}(u(s)-v(s))\|ds>R\}.
\end{split}
\end{equation*}
For any $n\in \mathbb{N},$ we define
\begin{equation*}
\begin{split}
B_{n}:=\{\|u(n)-v(n)\|^{2}+\int_{n}^{n+1}\|P_{N}(u(s)-v(s))\|^{2}ds>\frac{1}{n^{2}}\}.
\end{split}
\end{equation*}
\par
We divide the proof into several steps.
\par
\textbf{Step 1.} We claim that there is some $m^{*}$ sufficiently large such that
\begin{equation}\label{57}
\begin{split}
\mathbb{P}(\bigcap_{n=m^{*}}^{\infty}B_{n}^{c})>\frac{3}{4}.
\end{split}
\end{equation}
\par
Indeed, we set $B:=\bigcap_{m=1}^{\infty}\bigcup _{n=m}^{\infty}B_{n}.$ We take $N\geq N_{1}$($N_{1}$ is in Proposition \ref{FPE})
and set $\lambda=$ in the nudged equation \eqref{Mz}. It follows from \eqref{2}, the Chebyshev inequality, the Fubini theorem that
\begin{equation*}
\begin{split}
\mathbb{P}(B_{n}\cap \{\tau_{R,\beta}=+\infty\})
& \leq n^{2}\mathbb{E}[\mathbf{1}_{\tau_{R,\beta}=+\infty}(\|u(n)-v(n)\|^{2}+\int_{n}^{n+1}\|u(s)-v(s)\|^{2}ds)]\\
& \leq Cn^{2}e^{-Cn},
\end{split}
\end{equation*}
this implies that $\sum\limits_{n=1}^{\infty}\mathbb{P}(B_{n}\cap \{\tau_{R,\beta}=+\infty\})<\infty,$ thus, we have
$\mathbb{P}(B\cap \{\tau_{R,\beta}=+\infty\})=0.$
\par
We take some suitable $\beta$ fixed, according to \eqref{7}, we have $\mathbb{P}(B\cap \{\tau_{R,\beta}<+\infty\})\leq \mathbb{P}(\{\tau_{R,\beta}<+\infty\})=o(R^{-1}).$ This implies that
\begin{equation*}
\begin{split}
\mathbb{P}(B)=\mathbb{P}(B\cap \{\tau_{R,\beta}=+\infty\})+\mathbb{P}(B\cap \{\tau_{R,\beta}<+\infty\})
=o(R^{-1}),
\end{split}
\end{equation*}
then $\mathbb{P}(B^{c})=1,$ from the
continuity from below, we can thus take $m^{*}$ sufficiently large so that \eqref{57} holds.
\par
\textbf{Step 2.} We claim that there is some $R^{*}$ sufficiently large such that
\begin{equation}\label{58}
\begin{split}
\mathbb{P}(E_{R^{*},m^{*}}^{c}\cap\bigcap_{n=m^{*}}^{\infty}B_{n}^{c})\geq \frac{1}{2}.
\end{split}
\end{equation}
\par
Indeed, it follows from Lemma \ref{L3}, Lemma \ref{L4} and the Chebyshev inequality that
\begin{equation*}
\begin{split}
\mathbb{P}(E_{R,m^{*}})\leq \frac{\mathbb{E}[\int_{0}^{m^{*}}\|P_{N}(u(s)-v(s))\|^{2}ds]}{R}\leq \frac{C(u_{0},v_{0})m^{*}}{R},
\end{split}
\end{equation*}
by taking $R^{*}$ large enough, we have $\mathbb{P}(E_{R^{*},m^{*}}^{c})\geq \frac{3}{4}.$
Moreover, we have \eqref{58}.
\par
\textbf{Step 3.} We finish Proposition \ref{pro15}.
\par
On the set $E_{R^{*},m^{*}}^{c}\cap\bigcap_{n=m^{*}}^{\infty}B_{n}^{c},$ it holds that
\begin{equation*}
\begin{split}
\int_{0}^{\infty}\|P_{N}(u(s)-v(s))\|^{2}ds
&=\int_{0}^{m^{*}}\|P_{N}(u(s)-v(s))\|^{2}ds+\sum\limits_{n=m^{*}}^{\infty}\|P_{N}(u(s)-v(s))\|^{2}ds\\
&\leq R^{*}+\sum\limits_{n=m^{*}}^{\infty}n^{-2}:=K.
\end{split}
\end{equation*}
This implies that $$E_{R^{*},m^{*}}^{c}\cap\bigcap_{n=m^{*}}^{\infty}B_{n}^{c}\subset \{\sigma_{K}=\infty\}.$$ On the other hand, for any $m,$
$$\bigcap_{n=m}^{\infty}B_{n}^{c}\subset \{\lim\limits_{n\rightarrow\infty}\|u(n)-\tilde{v}(n)\|=0\}.$$
Thus, by \eqref{58}, we have
\begin{equation*}
\begin{split}
\mathbb{P}\{\lim\limits_{n\rightarrow\infty}\|u(n)-\tilde{v}(n)\|=0\}
&\geq\mathbb{P}\{(\lim\limits_{n\rightarrow\infty}\|u(n)-\tilde{v}(n)\|=0)\cap (\sigma_{K}=\infty)\}\\
&=\mathbb{P}\{(\lim\limits_{n\rightarrow\infty}\|u(n)-v(n)\|=0)\cap (\sigma_{K}=\infty)\}\\
&\geq\mathbb{P}(E_{R^{*},m^{*}}^{c}\cap\bigcap_{n=m^{*}}^{\infty}B_{n}^{c})\geq \frac{1}{2}.
\end{split}
\end{equation*}
\end{proof}

\begin{proof}[Proof of Theorem \ref{MT2}]
For any $u_{0},v_{0}\in H,$ we define $\tilde{\xi}_{u_{0},v_{0}}=\mathcal{D}(\{(u(n),\tilde{v}(n))\}_{n\in \mathbb{N}}),$ where $y$
and $\tilde{v}$ are the solutions to \eqref{My} and \eqref{Mz1} with corresponding initial data $u_{0},v_{0},$ respectively.
We can see $\tilde{\xi}_{u_{0},v_{0}}$ is a measure on $H^{\mathbb{N}}\times H^{\mathbb{N}}.$ By Proposition \ref{pro16}, $\pi_{2}(\tilde{\xi}_{u_{0},v_{0}})\sim \mathbb{P}_{z_{0}},$
then $\tilde{\xi}_{u_{0},v_{0}}\in \mathcal{\tilde{C}}(\mathbb{P}_{u_{0}},\mathbb{P}_{v_{0}}).$ By Proposition \ref{pro15}, $\tilde{\xi}_{u_{0},v_{0}}(D)=\mathbb{P}(\lim\limits_{n\rightarrow\infty}\|u(n)-\tilde{v}(n)\|=0)>0.$ According to Theorem \ref{AC1}, we can prove Theorem \ref{MT2}.
\end{proof}
\subsection{Proof of Theorem \ref{MT3}}
Let Condition (g3) hold for some $M\geq N_{0}.$ Let $u$ and $v$ be the solutions to \eqref{My} and \eqref{Mz} with initial dates $u_{0},v_{0}\in H,$
respectively. The proof of Theorem \ref{MT3} is divided into several steps.
\par
\textbf{Step 1.} For any $u_{0},v_{0}\in H,$ we define $\xi_{u_{0},v_{0}}:=\mathcal{D}(\{(u(n),v(n))\}_{n\in \mathbb{N}}),$ then $\xi_{u_{0},v_{0}}$ is a measure on $H^{\mathbb{N}}\times H^{\mathbb{N}}$ and $\pi_{1}(\xi_{u_{0},v_{0}})\sim \mathbb{P}_{u_{0}}.$ If we have
\begin{equation}\label{60}
\begin{split}
\pi_{2}(\xi_{u_{0},v_{0}})\sim \mathbb{P}_{v_{0}},
\end{split}
\end{equation}
then $\xi_{u_{0},v_{0}}\in \mathcal{\tilde{C}}(\mathbb{P}_{u_{0}},\mathbb{P}_{v_{0}}).$
\par
Indeed, we know the solution to \textit{nudged equation} is the solution to the following system
\begin{eqnarray}\label{}
\begin{array}{l}
\left\{
\begin{array}{llll}
dv+[Av+B(v)]dt=hdt+g(v)d\hat{W}
\\ v(x,0)=v_{0}(x)
\end{array}
\right.
\end{array}
\end{eqnarray}
where $\tilde{W}:=W(t)+\int_{0}^{t}h(s)ds.$
\begin{equation*}
\begin{split}
\mathbb{E}\int_{0}^{\infty}\|h(t)\|_{U}^{2}dt\leq \lambda^{2}(\sup\limits_{u\in H}\|f(u)\|_{L(H, U)}^{2})\mathbb{E}\int_{0}^{\infty}\|u(t)-v(t)\|^{2}dt.
\end{split}
\end{equation*}
According to Theorem \ref{FPE}, we have
\begin{equation*}
\begin{split}
\mathbb{E}\|u(t)-v(t)\|^{2}\leq
\left\{
\begin{array}{lll}
e^{C(1+\|u_{0}\|^{2p}+\|v_{0}\|^{2p})}e^{-Ct} \\
e^{C(1+\|u_{0}\|^{2p}+\|v_{0}\|^{2p})}\frac{1}{t^{\frac{p}{4}-\frac{1}{2}}}
\end{array}
\right.
\begin{array}{lll}
under~(g2)(i),\\
under~(g2)(ii)~or~(g2)(iii),
\end{array}
\end{split}
\end{equation*}
for any $p$ in \eqref{54}. This implies that
\begin{equation*}
\begin{split}
\mathbb{E}\int_{0}^{\infty}\|u(t)-v(t)\|^{2}dt=\mathbb{E}\int_{0}^{1}\|u(t)-v(t)\|^{2}dt+\mathbb{E}\int_{1}^{\infty}\|u(t)-v(t)\|^{2}dt<\infty.
\end{split}
\end{equation*}
By the Girsanov Theorem, the law of $\hat{W}$ is absolutely continuous w.r.t. the
law of $W$. In turn, the law of the solution $z$ to the nudged equation \eqref{Mz} is
absolutely continuous w.r.t. the law of the solution $u$ to equation \eqref{My} with
initial data $v_{0}$, as measures on $C([0,\infty);H).$ This proves \eqref{60}.
\par
\textbf{Step 2.} We prove
\begin{equation*}
\begin{split}
\lim\limits_{n\rightarrow\infty}\xi_{u_{0},v_{0}}(D_{\varepsilon}^{n})=\lim\limits_{n\rightarrow\infty}\mathbb{P}(\|u(n)-\tilde{v}(n)\|\leq \varepsilon)=1.
\end{split}
\end{equation*}
\par
Indeed, noting the fact
$\mathbb{P}(\|u(n)-\tilde{v}(n)\|> \varepsilon)\leq \frac{1}{\varepsilon^{2}}\mathbb{E}\|u(n)-\tilde{v}(n)\|^{2},$
with the help of Theorem \ref{FPE}, we have
$\lim\limits_{n\rightarrow\infty}\mathbb{P}(\|u(n)-\tilde{v}(n)\|> \varepsilon)=0.$
\par
\textbf{Step 3.}
Since the assumptions of Theorem \ref{AC2} are verified, according to Theorem \ref{AC2}, we can prove Theorem \ref{MT3}.
\par
This completes the proof.

\par~~
\par~~
\par

\noindent \footnotesize {\bf Acknowledgements.}
\par
This work was started when Peng Gao was visiting Institute of Mathematical Sciences in NYU-ECNU, he thanks to Professor Nersesyan V for his invitation
and the institute for its hospitality.
Peng Gao thanks Professor Kuksin S for his invitation to Universit\'{e} Paris Cit\'{e} and for his hospitality.
Peng Gao would like to thank the financial support of the China Scholarship Council (No. 202406620219).
This work is supported by
Natural Science Foundation of Jilin Province (Grant No. YDZJ202201ZYTS306), NSFC
(Grant No. 12371188) and the Fundamental Research Funds for the Central Universities (Grant
No. 135122006).
\par~~
\par
\noindent \footnotesize {\bf Data Availability.} \par
Data sharing not applicable to this article as no data sets were generated or analysed during the current study.
\par~~
\par
\noindent \footnotesize {\bf Competing Interests.} \par
The author declares that there is no conflict of interest.
\par
\par
\par

\noindent\baselineskip 6pt \renewcommand{\baselinestretch}{0.9}

\end{document}